\documentclass[reqno,11pt]{amsart}

\usepackage[margin=1in]{geometry}
\usepackage{amsmath,amssymb,amsthm}
\usepackage{graphicx}
\usepackage{xcolor}

\usepackage{cite}
\usepackage{hyperref}

\numberwithin{equation}{section}

\theoremstyle{plain}
\newtheorem{theorem}{Theorem}[section]
\newtheorem{proposition}[theorem]{Proposition}
\newtheorem{lemma}[theorem]{Lemma}
\newtheorem{corollary}[theorem]{Corollary}

\theoremstyle{definition}
\newtheorem{definition}[theorem]{Definition}

\newtheorem{remark}[theorem]{Remark}

\graphicspath{{images/}}

\begin{document}

\title{On a metric view of the polynomial shift locus}

\author{Yan Mary He}
\address{Department of Mathematics, University of Oklahoma, Norman, OK 73019}
\email{he@ou.edu}

\author{Hongming Nie}
\address{Los Angeles, CA, 91755}
\email{hongming.i.nie@gmail.com}
\date{\today}

\begin{abstract}
We introduce a path metric $\rho_D$ on the shift locus
$\mathcal S_D$ of degree $D\geq2$ complex polynomials by connecting complex
dynamics to geometric group theory.
We encode the escape and
deformation data of polynomial basins by metric graphs and use the
thermodynamic geometry of the resulting metric graph space. We prove that
$(\mathcal S_D,\rho_D)$ is complete if and only if $D=2$. For
$D\geq3$, we show that the positive-height part of the metric boundary is naturally
homeomorphic to the projectivized DeMarco--Pilgrim space
$\mathbb P\mathcal{ST}_D^*$. Our result provides a new way of understanding this space which encodes the topology of the shift locus and is defined abstractly as a monotone-light quotient.
\end{abstract}

\maketitle

\section{Introduction}
\label{sec:intro}

Fix an integer $D\geq2$. Let $\mathcal M_D$ denote the moduli space of
complex polynomials of degree $D$, up to affine conjugacy. The
\emph{shift locus} $\mathcal S_D\subset\mathcal M_D$ consists of those
conjugacy classes $[f]$ for which every critical point escapes to
infinity. For such a polynomial, the dynamics on the Julia set is
conjugate to the one-sided full shift on $D$ symbols.

When $D=2$, the shift locus is the complement of the Mandelbrot set.
Since the Mandelbrot set is connected \cite{DouadyHubbard},
$\mathcal S_2$ is homotopy equivalent to a circle. In higher degree,
the topology and deformation theory of the shift locus become
considerably richer and have been studied from a variety of
perspectives; see, for example,
\cite{Blanchard91,Branner88,Branner92,Calegari21,Calegari22,
DeMarco12,DeMarco11,DeMarcoP11,DeMarco17,DeMarco11B,
DeMarco10,Goldberg94,Kiwi05}.

The aim of this paper is to introduce a new {\it geometric} perspective for the shift locus by connecting complex dynamics to geometric group theory. More specifically, we construct a path metric $\rho_D$ on $\mathcal S_D$ using the geometry of Culler--Vogtmann outer space \cite{CullerVogtmann86} and the thermodynamic geometry of metric graphs developed by Aougab--Clay--Rieck \cite{Aougab23}, Pollicott--Sharp \cite{Pollicott14}, and Kao \cite{Kao17}.
The basin of infinity of a polynomial in the shift locus
comes equipped with critical heights, which measure the escape rates
of the critical points, together with natural stretching and twisting
deformations. We encode this height and twist data by a length
function on a rose graph and use the entropy metric on the
corresponding metric graph space to define $\rho_D$.

This construction 
provides a geometric realization of the space $\mathbb P\mathcal{ST}_D^*$ introduced by DeMarco--Pilgrim \cite{DeMarco11}. The space $\mathbb P\mathcal{ST}_D^*$ records the
connected components of critical height fibers and carries a canonical
locally finite simplicial structure, thereby retaining important
topological information about the shift locus. The original definition of $\mathbb P\mathcal{ST}_D^*$, however,
arises from the monotone--light factorization of the critical heights
map and presents the space as an abstract topological quotient. When $D\geq3$, we identify a distinguished subset of the
metric boundary of $(\mathcal S_D,\rho_D)$ with
$\mathbb P \mathcal{ST}_D^*$. The boundary metric therefore induces a
metric on $\mathbb P \mathcal{ST}_D^*$ and gives this quotient a
geometric realization. This viewpoint complements the
topological classification of polynomial basins in
\cite{DeMarcoP11}, the combinatorial descriptions in
\cite{DeMarco17}, and the computational study in \cite{DeMarco10}.

\subsection*{The entropy metric on the shift locus}

The construction of the metric stems from the natural deformation
theory of a polynomial basin. At a generic point of the shift locus
(see Definition \ref{def:generic}), the Branner--Hubbard wringing
deformation gives a canonical splitting of the tangent space into two
complementary directions: stretching changes the critical escape
rates, while twisting changes the gluing of the fundamental annuli
without changing the critical heights.

We encode these two types of deformation by a rose graph with
$2D-2$ edges. The first $D-1$ edges record the critical height data (i.e., escape rates for the $D-1$ critical points counted with multiplicity),
while the remaining $D-1$ edges record infinitesimal twisting. The
space of length functions on this rose graph carries the entropy metric
studied by Aougab--Clay--Rieck \cite{Aougab23},
Pollicott--Sharp \cite{Pollicott14}, and Kao \cite{Kao17}. Pulling
this metric back separately along the height and twist directions
gives two norms $\|\cdot\|_H$ and $\|\cdot\|_T$.

More precisely, at a generic point $[f]\in\mathcal S_D^g$,
independent wringing gives an intrinsic splitting
$
T_{[f]}\mathcal S_D
=
H_{[f]}\oplus T_{[f]}$.
For $v=v_H+v_T$ according to this splitting, we define the
\emph{split entropy norm} by
$\|v\|_{\mathcal S_D} := \|v_H\|_H+\|v_T\|_T.$
We then define $\rho_D$ as the associated path metric, allowing
piecewise smooth curves in the generic locus together with
pure-height access curves to the non-generic locus.

\begin{proposition}
\label{thm:metric}
The split entropy norm defines a path distance
$\rho_D$ on $\mathcal S_D$. More precisely, for
$[f_1],[f_2]\in\mathcal S_D$, let
$
\rho_D([f_1],[f_2])
:=
\inf_\gamma L_D(\gamma),
$
where the infimum is taken over admissible curves joining
$[f_1]$ to $[f_2]$, and
$L_D(\gamma) := \int \|\dot\gamma(t)\|_{\mathcal S_D}\,dt.$
Then $\rho_D$ is a distance function on $\mathcal S_D$.
\end{proposition}

We refer to Definition \ref{def:adm-curves} for the definition of admissible curves.
The split norm is useful as it prevents cancellation between
height and twist motion. The associated path metric also admits an equivalent description in
terms of pure-height and pure-twist motion: every generic curve can be
approximated, with arbitrarily small increase in length, by a
staircase consisting alternately of pure-height and pure-twist
pieces. Consequently, the
same path metric is obtained by restricting the infimum defining
$\rho_D$ to piecewise pure-height and pure-twist curves; see Proposition \ref{prop:split-pure-equivalence}.

The proof of Proposition \ref{thm:metric} has two ingredients. First,
the independent wringing deformation gives local height--twist product
coordinates on the generic locus. Using the closed stretching simplices of DeMarco--Pilgrim, we show
that points on the non-generic locus can still be reached by
finite-length pure-height curves; see Lemma
\ref{lem:finite-length-access}. Together with the local
height--twist geometry, this gives finite-length admissible paths
between arbitrary points of the shift locus; see Lemma
\ref{lem:finite-admissible-connectivity}.

Second, to prove that distinct points cannot have zero distance, we
use critical B\"ottcher coordinates. Their absolute values record the
critical heights, while their arguments detect the angular information
changed by twisting. We prove that, on a neighborhood of every point of $\mathcal S_D$,
variation of the B\"ottcher coordinates is bounded above by split
entropy length; see Proposition \ref{prop:local-bottcher-control}.
The relevant critical B\"ottcher-coordinate map is obtained by
evaluating a normalized B\"ottcher coordinate at sufficiently high
iterates of the marked critical points. Since this map has finite
fibers, every curve leaving a sufficiently small neighborhood has uniformly
positive length; see Corollary \ref{cor:local-lower-separation}. Hence $\rho_D$ separates
points. Proposition \ref{prop:metric-topology} also shows that
$\rho_D$ induces the usual topology on $\mathcal S_D$. 

\medskip

The geometry of the metric can already be seen in degree two. In this case there is
one critical height, so the level sets of the escape rate function are
pure-twist curves, while changing the escape rate gives the
pure-height direction. 
With the chosen twist normalization, these level sets become short
near both ends of the shift locus: close to the boundary of the
Mandelbrot set and far toward infinity. This behavior is
illustrated in Figure \ref{fig:1}; see Section \ref{subsec:infinite-distance} for more details.

\begin{figure}[h]
\centering
\begin{minipage}[b]{0.33\textwidth}
\centering
\includegraphics[width=\textwidth]{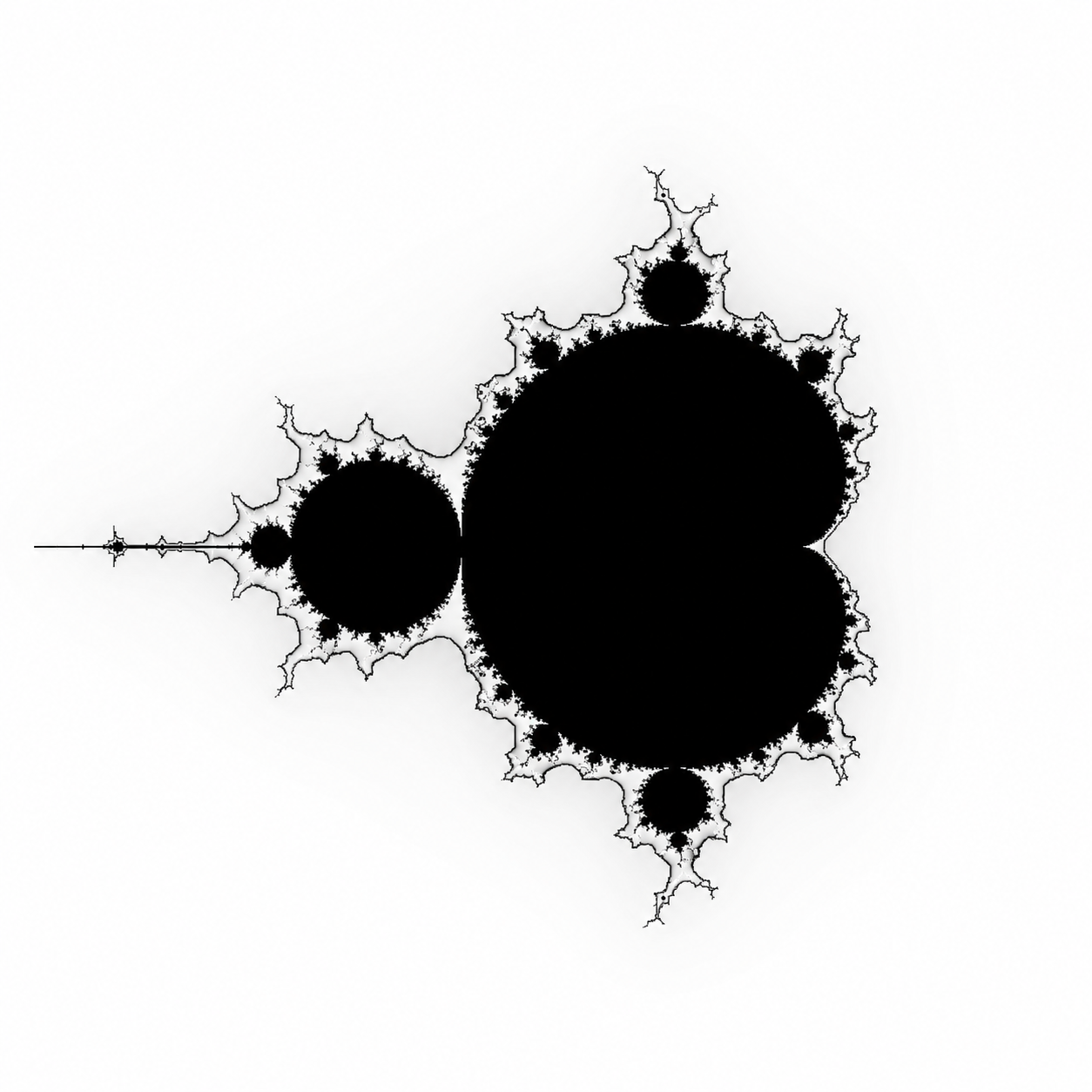}
\end{minipage}
\hspace{0.04\textwidth}
\begin{minipage}[b]{0.33\textwidth}
\centering
\includegraphics[width=\textwidth]{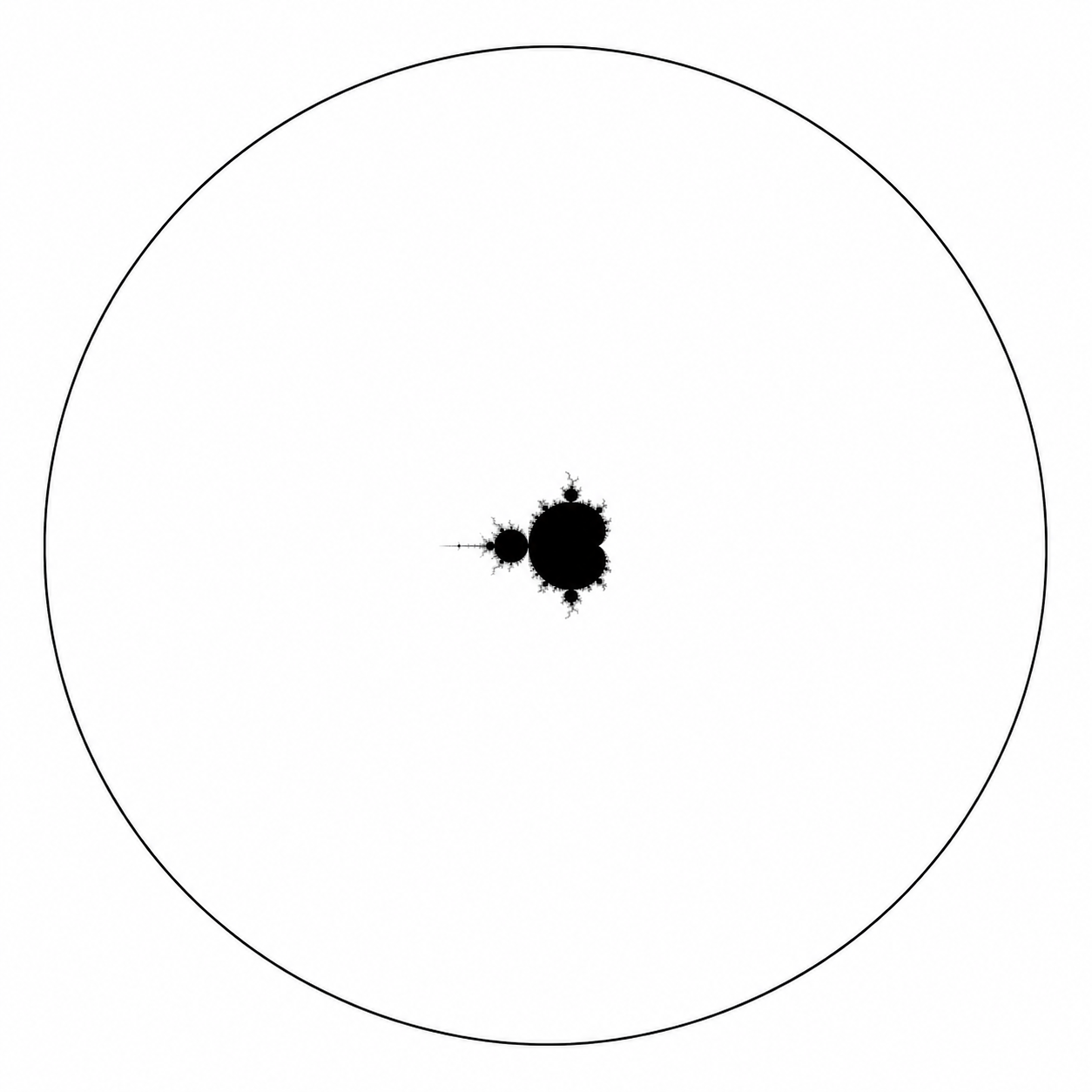}
\end{minipage}
\caption{
For $c\in\mathbb C$, let $f_c(z)=z^2+c$, and let $G_{f_c}(0)$ be
the escape rate of its critical point. In
$\mathcal S_2=\{c\in\mathbb C:f_c^n(0)\to\infty\}$, the
$\rho_2$-length of the level set
$\{c\in\mathcal S_2:G_{f_c}(0)=x\}$ tends to zero both as
$x\to0$ and as $x\to+\infty$.
}
\label{fig:1}
\end{figure}

\subsection*{The metric completion and the DeMarco--Pilgrim space}

We next study the large scale geometry and metric completion of
$(\mathcal S_D,\rho_D)$. 

Given a polynomial $f$, let
$G_f$ denote its Green function, and list the values $G_f(c)$, with
$c$ ranging over the critical points counted with multiplicity, in
decreasing order, i.e.,
$
h_1(f)\geq\cdots\geq h_{D-1}(f)\geq0.
$
The resulting map
$\mathcal G\colon\mathcal M_D\to \mathcal H_D$ given by $\mathcal G([f]) := \bigl(h_1(f),\ldots,h_{D-1}(f)\bigr)$
is the {\it critical heights map }introduced and studied by
DeMarco--Pilgrim \cite{DeMarco11}.

A
fiber of $\mathcal G$ may have several connected components. DeMarco--Pilgrim therefore introduced a quotient
$\mathcal T_D^*$ by collapsing each connected component of a
critical height fiber to a point. This gives the monotone--light
factorization
$\mathcal M_D \to \mathcal T_D^* \to \mathcal H_D,$
where the first map is {\it monotone} (i.e.,
it has connected fibers) and the second map is {\it light} (i.e., it
has totally
disconnected fibers); see
\cite[Theorem 1.3 and Corollary 1.4]{DeMarco11}.

Global stretching multiplies every critical height by the same
positive factor and induces compatible cone structures on these
spaces. Restricting to the shift locus and quotienting by this
stretching action gives the projectivized DeMarco--Pilgrim space
$\mathbb P\mathcal{ST}_D^*$. This space records not only the relative escape
rates of the critical points, but also the connected component of the
corresponding critical height fiber (i.e., fiber of $\mathcal{G}$). DeMarco--Pilgrim proved that
$\mathbb P\mathcal{ST}_D^*$ carries a canonical locally finite
simplicial structure, whose top-dimensional open simplices correspond
to structurally stable topological conjugacy classes of polynomial
basins; see \cite[Theorems 1.7 and 1.8]{DeMarco11}.

Our second main result shows that this topological quotient arises naturally
from the metric geometry of the shift locus. Let
$\hat{\mathcal S}_D$ be the metric completion of
$(\mathcal S_D,\rho_D)$. We denote by
$\mathcal C_D^{\infty,+}$ the part of its boundary represented by
$\rho_D$-Cauchy sequences $\{[f_k]\}$ for which the largest critical
height tends to infinity while all relative critical heights converge
to positive limits; that is,
$h_1(f_k)\to +\infty$ and 
$\frac{h_j(f_k)}{h_1(f_k)}
\to  r_j>0$ for $2\leq j\leq D-1.$
These are precisely the degenerations in which the entire
critical height configuration escapes to infinity without any
critical height becoming negligible relative to the largest one. We call $\mathcal C_D^{\infty,+}$ the positive-height part of the boundary.

\begin{theorem}
\label{thm:completion}
The metric space $(\mathcal S_D,\rho_D)$ is complete if and only if
$D=2$. For $D\geq3$, there is a natural homeomorphism
$
\mathcal C_D^{\infty,+}
\cong
\mathbb P\mathcal{ST}_D^*.
$
In particular, $\mathbb P\mathcal{ST}_D^*$ occurs naturally as a
distinguished subset of the metric boundary of $(\mathcal S_D,\rho_D)$.
\end{theorem}

Theorem \ref{thm:completion} gives a geometric realization of the space
$\mathbb P\mathcal{ST}_D^*$. It shows that the equivalence relation defining this space is recovered naturally from finite length degeneration in the entropy metric: sequences converging to the same projective model become asymptotically indistinguishable, whereas sequences converging to distinct models remain uniformly separated.

The proof is governed by two opposing mechanisms. The first is
\emph{collapse}: far out in the shift locus, twisting becomes
inexpensive, radial stretching has asymptotically negligible length,
and connected components of critical height fibers shrink in the
metric. Consequently, degenerations converging to the same point of
$\mathbb P\mathcal{ST}_D^*$ determine the same boundary point; see Proposition
\ref{prop:asymptotic-model-collapse}. 
The
second mechanism is \emph{separation}: on compact positive-height
regions at large height, split entropy length controls projected path
length in $\mathbb P\mathcal{ST}_D^*$. By also excluding short paths
that leave these regions or descend to finite height, we separate
distinct projective models, including models with the same projective
critical heights; see Proposition \ref{prop:positive-model-separation}.

The contrast between degree two and higher degree is already visible
in the asymptotic geometry of the associated rose graphs. In degree
two, every sequence leaving compact subsets has a subsequence along
which $h_1\to0$ or $h_1\to+\infty$. Both limits give singular rose
degenerations at infinite entropy distance, yielding completeness. When $D\geq3$, several rose edges can remain
comparable while another becomes arbitrarily long. Such degenerations
can occur at finite distance and produce the positive-height boundary
$\mathcal C_D^{\infty,+}$.

\subsection*{Comparison with the polynomial-tree compactification}

As an immediate consequence of Theorem \ref{thm:completion}, the
positive-height part of the boundary $\mathcal C_D^{\infty,+}$ also admits a natural projection to
the DeMarco--McMullen polynomial-tree boundary \cite{DeMarco08}. Recall that
$q\colon \mathbb P\mathcal{ST}_D^* \to \mathbb P\mathcal{ST}_D$
is the natural map obtained by replacing a polynomial basin by its
metrized quotient tree.

\begin{corollary}
\label{cor:boundary-tree-comparison}
Let $D\geq3$. Under the homeomorphism
$
\mathcal C_D^{\infty,+}
\cong
\mathbb P\mathcal{ST}_D^*
$
from Theorem \ref{thm:completion}, the projection $q$ induces a
continuous map
$
\Theta\colon
\mathcal C_D^{\infty,+}
\to 
\mathbb P\mathcal{ST}_D.
$
If a $\rho_D$-Cauchy sequence $\{[f_k]\}\subset\mathcal S_D$
represents $x\in\mathcal C_D^{\infty,+}$, then the projectivized
polynomial trees of $f_k$ converge to $\Theta(x)$. Moreover, for every
$\tau\in\mathbb P\mathcal{ST}_D$, the fiber
$\Theta^{-1}(\tau)$ is naturally identified with
$q^{-1}(\tau)$.
\end{corollary}

Therefore $\mathcal C_D^{\infty,+}$ refines the polynomial-tree
boundary in the sense that a projectivized polynomial tree records the metrized quotient
tree of a degenerating basin, while the metric boundary retains the
finer information distinguished by the DeMarco--Pilgrim quotient.

\medskip

We can also characterize which degenerations toward the
DeMarco--McMullen boundary escape to infinite distance in the metric.
This gives a direct comparison between the large-scale geometry of
$(\mathcal S_D,\rho_D)$ and the polynomial-tree compactification.

\begin{theorem}
\label{thm:tree-infinite-distance}
Let $\{[f_k]\}\subset\mathcal S_D$ satisfy
$h_1(f_k)\to+\infty$.

If $D=2$, then for every fixed $[f_0]\in\mathcal S_2$, we have
$
\rho_2([f_0],[f_k])
\to 
+\infty.
$

If $D\geq3$, then for every fixed $[f_0]\in\mathcal S_D$,
$
\rho_D([f_0],[f_k])
\to 
+\infty
$
if and only if
$
\frac{h_{D-1}(f_k)}{h_{D-2}(f_k)}
\to 
0.
$
\end{theorem}

Theorem \ref{thm:tree-infinite-distance} shows that infinite metric
distance is detected by a separation of scales among the smallest
critical heights. In the terminology of Section
\ref{subsec:entropy-base-length-functions}, these are precisely the
degenerations for which the associated rose metrics develop a
singular minimal scale. When the two smallest critical heights remain
comparable, at least two rose edges survive at the minimal scale and
the degeneration remains at bounded $\rho_D$-distance along a
subsequence. 
Thus the metric distinguishes singular tree degenerations at infinite
distance from non-singular degenerations at bounded distance.

\subsection*{Related work}

A foundational example of using thermodynamic formalism to construct
geometric structures on deformation spaces is the pressure-theoretic
interpretation of the Weil--Petersson metric. McMullen
\cite{McMullen08} showed that the
Weil--Petersson metric on Teichm\"uller spaces can be recovered, up to scale, from a pressure
form; see also \cite{Bridgeman10}. This perspective has since been developed in several settings,
including quasi-Fuchsian spaces, deformation spaces
of Anosov representations, moduli spaces of metric graphs, and
Culler--Vogtmann outer space \cite{CullerVogtmann86}; see, for example,
\cite{Bridgeman08,Bridgeman10,BCLS15,McMullen08,
Pollicott14,Kao17,Xu19}.

Through Sullivan's dictionary, complex dynamics provides a natural
setting where one can seek analogous geometric structures. McMullen
\cite{McMullen08} constructed a pressure metric on the space
$\mathcal B_D$ of expanding Blaschke products. Ivrii \cite{Ivrii}
studied this metric in degree two, proved that it is incomplete, and
described part of its metric completion. Further analytic and metric
properties of pressure forms on spaces of Blaschke products were
developed by the first author, Lee and Park in \cite{HLP23,HLP25}.

Thermodynamic metrics have also been constructed on other dynamically
defined parameter spaces. In our earlier paper \cite{HeNie23}, we introduced a
Riemannian metric on certain hyperbolic components in moduli spaces of
rational maps. Pressure path metrics on stable components of parabolic
and Misiurewicz families of polynomials were subsequently studied by Bianchi and the first author in
\cite{BH24,BH25}. These constructions are related, through Sullivan's
dictionary, to pressure metrics on deformation spaces in
Teichm\"uller theory and higher Teichm\"uller theory.

The present paper differs from these works in both its geometric input
and its objective. Rather than constructing a pressure form directly
from the derivative cocycle of a dynamical system, we encode the
height--twist deformation theory of polynomial basins by length
functions on a rose graph and import the entropy metric from the
geometry of metric graphs. This makes it possible to study the global
metric geometry of the entire shift locus and identify its positive-height
finite-length degenerations with the projectivized DeMarco--Pilgrim space
$\mathbb P\mathcal{ST}_D^*$.

\subsection*{Organization of the paper}

Section \ref{sec_prelim} reviews polynomial deformations, the
critical heights map and its factorizations, and the entropy metric on
rose graphs. In Section \ref{sec_3}, we define the path function $\rho_D$ on
the shift locus and prove Proposition \ref{thm:metric}. Section
\ref{sec:entropy-behavior} develops the entropy estimates needed to
analyze degenerating rose metrics. In Section
\ref{sec:metric-completion}, we study Cauchy sequences and prove
Theorem \ref{thm:completion}. Finally, in Section
\ref{sec:tree-comparison}, we prove Corollary \ref{cor:boundary-tree-comparison} and
Theorem \ref{thm:tree-infinite-distance}.

\subsection*{Acknowledgments}

The authors thank Kevin Pilgrim for helpful comments.

\subsection*{Notation}

For positive sequences $\{a_k\}$ and $\{b_k\}$, we write
$a_k=o(b_k)$ if $a_k/b_k\to0$, and $a_k=O(b_k)$ if
$a_k/b_k$ remains bounded above for all sufficiently large $k$. We
write $a_k\asymp b_k$ if $a_k/b_k$ remains bounded above and below by
positive constants. For a differentiable path $t\mapsto f(t)$, we
write $\dot f(t)=df(t)/dt$. For a finite set $S$, we denote its
cardinality by $|S|$.

\section{Preliminaries}
\label{sec_prelim}

In this section, we collect the background needed for the construction
of the metric. Section \ref{sec:poly} reviews critical heights,
independent wringing, and the DeMarco--Pilgrim factorization. Section
\ref{subsec:rose-entropy-metric} recalls the entropy metric on the
space of length functions on a rose graph, together with the formulas
and boundary behavior which will be used later.

\subsection{Polynomial dynamics}
\label{sec:poly}

\subsubsection{Critical heights}
\label{sec:height}

Let $\mathrm{Poly}_D$ be the space of monic centered polynomials of
degree $D\geq2$, and let $\mathcal M_D$ be its quotient by affine
conjugacy. For $f\in\mathrm{Poly}_D$, the basin of infinity is
$\Omega_f:=\{z:f^n(z)\to\infty\}$, and the escape-rate function is
$G_f(z) = \lim_{n\to\infty} \frac{1}{D^n}\log\max\{|f^n(z)|,1\}.$
It is continuous on $\mathbb C$, harmonic on $\Omega_f$, and satisfies
$G_f\circ f=DG_f$. In particular, $G_f(z)>0$ if and only if
$z\in\Omega_f$.

Let $\mathrm{Crit}_f$ denote the critical points of $f$, counted with
multiplicity, and write their escape rates in decreasing order as
$
h_1(f)\geq\cdots\geq h_{D-1}(f)\geq0.
$
Set
$\mathcal H_D := \{(h_1,\ldots,h_{D-1}): 0\leq h_{D-1}\leq\cdots\leq h_1<\infty\}.$
The \emph{critical heights map} of DeMarco--Pilgrim is the map
$
\mathcal G\colon\mathcal M_D\to \mathcal H_D$ given by
$[f]\mapsto(h_1(f),\ldots,h_{D-1}(f)).$
It is continuous, proper, and surjective by
\cite[Theorem 1.1]{DeMarco11}.

The \emph{shift locus} is defined as
$\mathcal S_D := \{[f]\in\mathcal M_D:h_{D-1}(f)>0\}.$

\begin{definition}
\label{def:generic}
A point $(h_1,\ldots,h_{D-1})$ of $\mathcal H_D$ is \emph{generic}
if every entry is positive and $h_i/h_j\notin D^{\mathbb Z}$ for
$i\neq j$. A point $[f]\in\mathcal M_D$ is \emph{generic} if
$\mathcal G([f])$ is generic. We denote the generic part of the shift
locus by $\mathcal S_D^g$.
\end{definition}

\subsubsection{Stretching and twisting}
\label{sec:def}

We recall the Branner--Hubbard wringing deformation and its independent
version, following \cite[Sections 5.1--5.2]{DeMarco11}; see also
\cite{Branner88,Branner92}. 
We write $\mathbb H:=\{t+is:t\in\mathbb R,\ s>0\}$.

The holomorphic form $2\partial G_f$ equips the basin of infinity with
a natural flat structure away from its zeros. Wringing changes this
flat structure by an affine deformation. Its stretching subgroup
multiplies every critical height by the same positive number, while
its twisting subgroup preserves all critical heights. We write
$s\cdot[f]$ for global stretching by $s>0$; thus
$\mathcal G(s\cdot[f])=s\,\mathcal G([f]).$

There is also an independent version of wringing. Two escaping
critical points belong to the same \emph{foliated equivalence class}
if the closures of their grand orbits in the basin coincide. If
$N_f$ is the number of such classes containing critical points, the
fundamental annulus is divided into $N_f$ fundamental subannuli
$A_1(f),\ldots,A_{N_f}(f)$. Each subannulus may be wrung
independently.

By \cite[Lemma 5.2]{DeMarco11}, on the locus where $N_f=N$,
independent wringing defines a continuous $\mathbb H^N$-action whose
orbit maps are holomorphic and whose stabilizers contain lattices of
translations. When $N=D-1$, the stabilizer is itself a translation
lattice. The real directions in the $j$-th factor give
\emph{twisting} in $A_j(f)$, while the imaginary directions give
\emph{independent stretching}.

For later use, we recall explicit elements of the stabilizer. Let
$m_j=m_j(f)$ be the modulus of $A_j(f)$, and let $d_j=d_j(f)$ be the
least common multiple of the degrees of the maps
$f^n\colon\widetilde A_j(f)\to A_j(f)$, where $\widetilde A_j(f)$ ranges
over components of iterated preimages of $A_j(f)$. The proof of
\cite[Lemma 5.2]{DeMarco11} shows that the stabilizer contains
\begin{equation}
\label{eq_dj/mj}
(0,\ldots,0,d_j/m_j,0,\ldots,0),
\qquad
1\leq j\leq N_f.
\end{equation}

At a generic point, the $D-1$ critical points belong to distinct
foliated equivalence classes, so $N_f=D-1$.

\begin{proposition}[{\cite[Theorem 1.2]{DeMarco11}}]
\label{prop:fiber_2.4}
The fiber of $\mathcal G$ over any generic point of $\mathcal H_D$
is a finite disjoint union of smooth real $(D-1)$-dimensional tori.
Each connected component is a twist-deformation orbit.
\end{proposition}

\begin{remark}
\label{rmk:twist}
Let $[f]\in\mathcal S_D^g$. Independent wringing gives a local
$\mathbb H^{D-1}$-action with orbit map
$\mathcal O_f \colon \mathbb H^{D-1}\to\mathcal S_D^g$. At
$\mathbf i=(i,\ldots,i)$, writing
$
T_{\mathbf i}\mathbb H^{D-1}
=
\mathbb R^{D-1}\oplus i\mathbb R^{D-1},
$
the real directions correspond to independent twisting and the
imaginary directions to independent stretching. Define
$$
H_{[f]}
:=
d\mathcal O_f|_{\mathbf i}
\bigl(i\mathbb R^{D-1}\bigr),
\qquad
T_{[f]}
:=
d\mathcal O_f|_{\mathbf i}
\bigl(\mathbb R^{D-1}\bigr).
$$

Twisting preserves the critical heights, so Proposition
\ref{prop:fiber_2.4} gives
$T_{[f]}=\ker d\mathcal G_{[f]}$. Independent stretching supplies
the complementary height directions. Thus
$T_{[f]}\mathcal S_D = H_{[f]}\oplus T_{[f]}.$
This splitting is intrinsic: changing the critical marking only
permutes the fundamental subannuli.
\end{remark}

\begin{lemma}
\label{lem:multistretching-height-isomorphism}
For every $[f]\in\mathcal S_D^g$, the differential of the critical
heights map restricts to an isomorphism
$d\mathcal G_{[f]}|_{H_{[f]}} \colon H_{[f]} \to T_{\mathcal G([f])}\mathcal H_D.$
\end{lemma}

\begin{proof}
Independent stretching of the fundamental subannuli gives a
real-analytic local section of the critical heights map over the
corresponding height stratum; see \cite[Lemma 5.3]{DeMarco11}. At a
generic point there are $D-1$ independent subannuli and the height
stratum is open in $\mathcal H_D$. Differentiating the local section
gives the result.
\end{proof}

\subsubsection{The DeMarco--Pilgrim and polynomial-tree quotients}

The critical heights map admits a monotone--light factorization
$\mathcal M_D \to \mathcal T_D^* \to \mathcal H_D,$
where $\mathcal T_D^*$ is obtained by collapsing each connected
component of a critical height fiber to a point. The first map has
connected fibers and the second has totally disconnected fibers; see
\cite[Theorem 1.3 and Corollary 1.4]{DeMarco11}.

There is a further factorization
$\mathcal M_D \to \mathcal T_D^* \to \mathcal T_D \to \mathcal H_D,$
where $\mathcal T_D$ is the space of metrized polynomial trees of
DeMarco--McMullen \cite{DeMarco08}. All maps are continuous, proper,
and surjective; over the shift locus the fibers of
$\mathcal T_D^*\to\mathcal T_D$ are finite
\cite[Theorem 1.3]{DeMarco11}. 
Write $\mathcal{SH}_D:=\{h\in\mathcal H_D:h_{D-1}>0\}$, and let
$\mathcal{ST}_D^*$ and $\mathcal{ST}_D$ be its preimages in
$\mathcal T_D^*$ and $\mathcal T_D$, respectively.

Global stretching descends equivariantly to these spaces and gives
cone structures away from the connectedness locus
\cite[Lemma 5.1 and Theorem 1.5]{DeMarco11}. We identify their
projectivizations with the slice of maximal critical height $1$.
Restricted to the shift locus, we obtain proper surjective maps
$\mathbb P\mathcal{ST}_D^* \to \mathbb P\mathcal{ST}_D \to \mathbb P\mathcal{SH}_D.$
The first space records connected components of critical height
fibers, the second records metrized polynomial trees, and the third
records projective critical heights.

By \cite[Theorems 1.7 and 1.8]{DeMarco11},
$\mathbb P\mathcal{ST}_D^*$, $\mathbb P\mathcal{ST}_D$, and
$\mathbb P\mathcal{SH}_D$ carry canonical locally finite simplicial
structures. The projection maps are simplicial, and on each
top-dimensional simplex the projective critical heights map is affine
and injective. The top-dimensional open simplices of
$\mathbb P\mathcal{ST}_D^*$ correspond to structurally stable
topological conjugacy classes of polynomial basins.

We write
$\hat\pi\colon \mathcal S_D \to \mathbb P\mathcal{ST}_D^*$
for the composition of the monotone quotient with projectivization.
The DeMarco--McMullen space $\mathbb P\mathcal T_D$ gives the
polynomial-tree boundary of $\mathcal M_D$; see
\cite[Theorem 1.5]{DeMarco08}.

\subsection{The entropy metric on rose graphs}
\label{subsec:rose-entropy-metric}

We now recall the thermodynamic metric on the space of length
functions on a rose graph. Our exposition follows Aougab--Clay--Rieck
\cite{Aougab23}; see also \cite{Kao17,Pollicott14}.

For $n\geq2$, let $\mathfrak R_n$ be the $n$-petal rose with unoriented edges
$e_1,\ldots,e_n$. Its space of positive length functions is
$\mathfrak M_{\mathfrak R_n} := \{\ell\colon\{e_1,\ldots,e_n\}\to\mathbb R_{>0}\} \cong\mathbb R_{>0}^n.$
The entropy of $\ell$ is the exponential growth rate of cyclically
reduced closed paths with respect to $\ell$; we denote it by
$\mathfrak h_{\mathfrak R_n}(\ell)$. The entropy function is
real analytic and satisfies
$\mathfrak h_{\mathfrak R_n}(a\ell) = a^{-1}\mathfrak h_{\mathfrak R_n}(\ell), a>0;$
see \cite[Lemma 3.4]{Aougab23}.

Let
$\mathfrak M_{\mathfrak R_n}^1 := \{\ell\in\mathfrak M_{\mathfrak R_n}: \mathfrak h_{\mathfrak R_n}(\ell)=1\}.$
On this hypersurface, the Hessian of the entropy defines a positive
definite Riemannian metric
$\langle v,w\rangle_{\mathfrak h,\mathfrak R_n} := D^2\mathfrak h_{\mathfrak R_n}(\ell)[v,w];$
see \cite[Theorem 3.7 and Section 4]{Aougab23}. We write
$\|v\|_{\mathfrak h,\mathfrak R_n}$ for the corresponding norm,
$L_{\mathfrak h,\mathfrak R_n}$ for path length, and
$\rho_{\mathfrak h,\mathfrak R_n}$ for the induced distance.

For a non-empty subset $S\subset\{1,\ldots,n\}$ and a coordinate
vector $u=(u_1,\ldots,u_n)$, write $u(S):=\sum_{j\in S}u_j$;
in particular, $\ell(S)=\sum_{j\in S}\ell(e_j)$. The entropy is characterized by
\begin{equation}
\label{eq:rose-entropy-equation}
1
=
\sum_{\emptyset\neq S\subset\{1,\ldots,n\}}
(2|S|-1)
\exp\bigl(
-\mathfrak h_{\mathfrak R_n}(\ell)\ell(S)
\bigr);
\end{equation}
see \cite[Propositions 7.3 and 7.5]{Aougab23}.

For the estimates below, it is convenient to use the following
positive formula for the entropy norm. If
$L\in\mathfrak M_{\mathfrak R_n}^1$ and
$v\in T_L\mathfrak M_{\mathfrak R_n}^1$, then
\begin{equation}
\label{eq:rose-entropy-norm-positive-formula}
\|v\|_{\mathfrak h,\mathfrak R_n}^2
=
\frac{
\displaystyle
\sum_{\emptyset\neq S\subset\{1,\ldots,n\}}
(2|S|-1)e^{-L(S)}v(S)^2
}{
\displaystyle
\sum_{\emptyset\neq S\subset\{1,\ldots,n\}}
(2|S|-1)e^{-L(S)}L(S)
}.
\end{equation}
Indeed, set
$h(t):=\mathfrak h_{\mathfrak R_n}(L+tv)$. Since $L$ has unit entropy
and $v$ is tangent to the unit-entropy hypersurface, we have
$h(0)=1$ and $h'(0)=0$. Differentiating
\eqref{eq:rose-entropy-equation} twice at $t=0$ gives
\eqref{eq:rose-entropy-norm-positive-formula}, because
$h''(0)=D^2\mathfrak h_{\mathfrak R_n}(L)[v,v]$. In particular, both
the numerator and denominator in
\eqref{eq:rose-entropy-norm-positive-formula} are non-negative, and
the denominator is strictly positive.

The following infinite-diameter estimate will be used in Section
\ref{sec:metric-completion}.

\begin{lemma}[{\cite[Proposition 7.14]{Aougab23}}]
\label{lem:rose-infinite-diameter}
Fix $n\geq2$, and let $\ell_0\in\mathfrak M_{\mathfrak R_n}^1$ be
the symmetric length function
$\ell_0(e_j)=\log(2n-1)$. For every $M>0$, there is
$\varepsilon>0$ such that, if
$\ell\in\mathfrak M_{\mathfrak R_n}^1$ satisfies
$\min_j\ell(e_j)\leq\varepsilon$, then we have
$
\rho_{\mathfrak h,\mathfrak R_n}(\ell,\ell_0)\geq M.
$
\end{lemma}

Finally, the entropy metric has a natural partial compactification
obtained by allowing edges to have infinite length. Set
$\hat{\mathfrak M}_{\mathfrak R_n} := (0,\infty]^n.$
With the convention $e^{-\infty}=0$, define the unit entropy partial
compactification by
$$
\hat{\mathfrak M}_{\mathfrak R_n}^{\,1}
:=
\left\{\ell\in\hat{\mathfrak M}_{\mathfrak R_n}:
\sum_{\emptyset\neq S\subset\{1,\ldots,n\}}
(2|S|-1)e^{-\ell(S)}=1\right\}.
$$
Every length function in this space has at least two finite edge
lengths; its entropy is the entropy of the subrose formed by those
edges.
The entropy norm extends continuously to a semi-norm on this boundary,
and the resulting metric space is the metric completion of
$\mathfrak M_{\mathfrak R_n}^1$; see
\cite[Theorem 1.2, Section 9.2, and Propositions 9.5--9.6]{Aougab23}.
Moreover, the metric topology on
$\hat{\mathfrak M}_{\mathfrak R_n}^{\,1}$ agrees with its coordinate
subspace topology in $(0,\infty]^n$ by
\cite[Proposition 9.6(3)]{Aougab23}.

\section{Entropy metric on the shift locus}
\label{sec_3}

In this section, we use critical heights and independent wringing to
define a path metric on the shift locus. Section
\ref{subsec:height-twist-geometry} describes the local height--twist
product structure on the generic locus. In Section
\ref{subsec:height-twist-speeds}, we use this splitting to define the
split entropy norm. Section \ref{subsec:closed-stretching-families}
recalls the extension of stretching families to the non-generic
locus. In Section \ref{subsec:admissible-curves}, we define
admissible curves and prove that the resulting path function is
finite. Finally, Sections \ref{subsec:critical-bottcher} and
\ref{subsec:shift-locus-metric} use critical B\"ottcher coordinates
to prove Proposition \ref{thm:metric}.

Throughout this section, let
$\mathfrak R:=\mathfrak R_{2D-2}$, and continue to use the notation of
Section \ref{subsec:rose-entropy-metric}.

\subsection{Height--twist geometry on the generic locus}
\label{subsec:height-twist-geometry}

Let $[f_0]\in\mathcal S_D^g$. After passing to a finite critically
marked cover of a neighborhood of $[f_0]$, label the critical points
as $c_1(f),\ldots,c_{D-1}(f)$ and set
$g_j(f):=G_f(c_j(f))$.

In the labelled critical height space, the non-generic locus is the
union of the hypersurfaces $g_i=D^ng_j$, where $i\neq j$ and
$n\in\mathbb Z$. A connected component of their complement is called
a \emph{generic chamber}. On a relatively compact marked
neighborhood, only finitely many such hypersurfaces occur.

After shrinking around $[f_0]$, we may assume that the neighborhood
lies in one generic chamber. The critical points then belong to
distinct foliated equivalence classes, and the fundamental subannuli
can be labelled consistently as
$A_1(f),\ldots,A_{D-1}(f)$.

\begin{lemma}
\label{lem:height-twist-chart}
After shrinking the marked neighborhood of $[f_0]$, there is a
real-analytic coordinate chart
$
\Psi\colon B_g\times B_\theta\to\mathcal S_D^g,
$
where $B_g\subset\mathbb R_{>0}^{D-1}$ and
$B_\theta\subset\mathbb R^{D-1}$ are open. The first coordinates are
the labelled critical heights $(g_1,\ldots,g_{D-1})$, and the second
coordinates are local lifts $(\theta_1,\ldots,\theta_{D-1})$ of the
independent twist parameters. Varying only the height variables gives
independent stretching, while varying only the twist variables gives
independent twisting.
\end{lemma}

\begin{proof}
By \cite[Lemma 5.2]{DeMarco11}, independent wringing defines a
real-analytic $\mathbb H^{D-1}$-action near $[f_0]$, with discrete
translation stabilizer. Let
$\mathcal O_{f_0}\colon\mathbb H^{D-1}\to\mathcal S_D^g$ be the orbit
map and write $w_j=\theta_j+is_j$.

Consider near $\mathbf i=(i,\ldots,i)$ the map
$$
w\mapsto
\bigl(
g_1(\mathcal O_{f_0}(w)),\ldots,
g_{D-1}(\mathcal O_{f_0}(w)),
\theta_1,\ldots,\theta_{D-1}
\bigr).
$$
Twisting preserves the critical heights, while Lemma
\ref{lem:multistretching-height-isomorphism} shows that the
differential of the critical heights map is an isomorphism on the
independent-stretching directions. The differential of the map above at $\mathbf i$ is therefore invertible, and the real-analytic
inverse function theorem gives the desired chart.
\end{proof}

The real coordinates in Lemma \ref{lem:height-twist-chart} are tied to
the fixed orbit map based at $[f_0]$ and should not be confused with
the intrinsic twist speeds at a varying point. We now record the
normalization that will be used throughout the paper. For a generic
marked polynomial $f$, let
$X_j(f) := \left.\frac{d}{dt}\right|_{t=0} ( i,\ldots,i,t+i,i,\ldots,i)\cdot f$
be the current-point unit twist vector in the $j$-th fundamental
subannulus. Thus every $v_T\in T_{[f]}$ has a unique expression
$
v_T=\sum_{j=1}^{D-1}\xi_jX_j(f).
$
We call $\xi_j$ the \emph{current-point twist speeds}.

This convention is compatible with the affine group law in the
Branner--Hubbard wringing coordinates. Indeed, identify
$t+is\in\mathbb H$ with the matrix
$
A(t,s)=\left(\begin{smallmatrix}1&t\\0&s\end{smallmatrix}\right)
$
as in \cite[Section 5.1]{DeMarco11}. If
$f=\mathcal O_{f_0}(w)$ with $w_j=\theta_j+is_j$, then
$
A(\varepsilon,1)A(\theta_j,s_j)
=
A(\theta_j+s_j\varepsilon,s_j).
$
Consequently,
$
X_j(f)=s_j\,\partial_{\theta_j}
$
in the fixed orbit chart, and hence
$
\xi_j=\dot\theta_j/s_j.
$
Likewise,
$$
A(0,1+\varepsilon)A(\theta_j,s_j)
=A(\theta_j,(1+\varepsilon)s_j),
$$
so the current-point stretching vector is $s_j\partial_{s_j}$.
Hence varying the height coordinates while fixing the twist coordinates
gives pure-height motion at every point of the chart.

In particular, the translation length $d_j(f)/m_j(f)$ from
\eqref{eq_dj/mj} is a period of the one-parameter twist orbit
$t\mapsto(i,\ldots,i,t+i,i,\ldots,i)\cdot f$
in the $j$-th factor.

\subsection{The split entropy norm}
\label{subsec:height-twist-speeds}

Recall that the ordered critical heights satisfy
$h_1(f)\geq\cdots\geq h_{D-1}(f)>0$. Define the \emph{base length
function} $\ell_B\colon\mathcal S_D\to\mathfrak M_{\mathfrak R}$ by
$$
\ell_B([f])
:=
\left(
h_1(f),
\frac{h_2(f)}{h_1(f)},\ldots,
\frac{h_{D-1}(f)}{h_1(f)},
1,\ldots,1
\right).
$$
The first coordinate records the overall escape scale, the next
$D-2$ coordinates record the projective critical heights, and the
last $D-1$ coordinates will be used infinitesimally to record
twisting.

Let $Q(\ell):=\mathfrak h_{\mathfrak R}(\ell)\ell$ be entropy
normalization and set $\hat\ell_B:=Q\circ\ell_B$.

For $v_H\in H_{[f]}$, define
$\|v_H\|_H:=\|d\hat\ell_B(v_H)\|_{\mathfrak h,\mathfrak R}$.
For a twist tangent vector $v_T=\sum_{j=1}^{D-1}\xi_jX_j(f)\in T_{[f]}$, set
$H([f]) := \max\left\{h_1(f),\frac{1}{h_{D-1}(f)}\right\}.$
Define $\delta\ell_f(v_T)\in
T_{\ell_B([f])}\mathfrak M_{\mathfrak R}$ by setting its first
$D-1$ coordinates equal to zero and
$$
\delta\ell_f(v_T)(e_{D-1+j})
=
\frac{\xi_j}
{H([f])\max\{1,d_j(f)/m_j(f)\}},
\qquad 1\leq j\leq D-1.
$$
We then set
$$
\|v_T\|_T
:=
\left\|
dQ_{\ell_B([f])}\bigl(\delta\ell_f(v_T)\bigr)
\right\|_{\mathfrak h,\mathfrak R}.
$$

The normalization in the twist directions is chosen relative to the
current-point periods in \eqref{eq_dj/mj}. 
In particular,
$\max\{1,d_j/m_j\}$ bounds the integral of $|\xi_j|$ along a monotone
complete turn in the $j$-th twist factor, while the factor
$H([f])=\max\{h_1(f),h_{D-1}(f)^{-1}\}$ makes twisting inexpensive
near both natural ends of the shift locus.

For $v\in T_{[f]}\mathcal S_D$, write uniquely
$v=v_H+v_T$ according to the splitting in Remark \ref{rmk:twist}.

\begin{definition}
\label{def:split-entropy-norm}
The \emph{split entropy norm} on the generic locus is
$\|v\|_{\mathcal S_D} := \|v_H\|_H+\|v_T\|_T.$
A smooth curve in $\mathcal S_D^g$ is called \emph{pure height} if
its tangent vectors lie in $H$, and \emph{pure twist} if its tangent
vectors lie in $T$.
\end{definition}

\begin{lemma}
\label{lem:split-entropy-norm}
The split entropy norm is a well-defined continuous norm on
$T\mathcal S_D^g$.
\end{lemma}

\begin{proof}
We first consider the height directions. By Lemma
\ref{lem:multistretching-height-isomorphism},
$d\mathcal G|_{H_{[f]}}$ is an isomorphism. Since
$h_1,h_2/h_1,\ldots,h_{D-1}/h_1$ determine the critical heights,
$d\ell_B$ is injective on $H_{[f]}$. By homogeneity of entropy, the
kernel of $dQ_\ell$ is the radial line $\mathbb R\ell$. A tangent
vector to the image of $\ell_B$ has its last $D-1$ coordinates equal
to zero, whereas those coordinates of $\ell_B([f])$ equal $1$.
Thus $d\hat\ell_B$ is injective on $H_{[f]}$, and
$\|\cdot\|_H$ is a norm.

For the twist directions, the coefficients $\xi_j$ are defined using
the current-point unit twist vectors $X_j(f)$ and hence do not depend
on a fixed-base orbit chart. A change of critical marking only
permutes the fundamental subannuli, and hence simultaneously permutes
$m_j,d_j$ and $\xi_j$. The entropy metric on the rose is invariant
under the corresponding permutation of edges, so $\|\cdot\|_T$ is
independent of all choices.

Moreover, $\delta\ell_f(v_T)=0$ if and only if $v_T=0$, and a
non-zero vector supported on the last $D-1$ coordinates cannot lie in
the radial kernel of $dQ$. Thus $\|\cdot\|_T$ is a norm. Continuity
follows in a marked height--twist chart: the current-point twist
vectors $X_j$ vary real-analytically, the integers $d_j$ are locally
constant, and $m_j$, $H$, and the entropy metric vary continuously.
The result follows from the direct-sum decomposition
$T\mathcal S_D^g=H\oplus T$.
\end{proof}

In a marked product chart, if
$v=v_H+v_T$ corresponds to $(\dot g,\dot\theta)$, set
$|v|_{\mathrm{coord}}:=|\dot g|+|\dot\theta|$.

\begin{lemma}
\label{lem:generic-local-comparison}
Let $K\Subset\mathcal S_D^g$ be contained in one marked height--twist
chart. There are constants $0<c_K\leq C_K$ such that
$c_K|v|_{\mathrm{coord}} \leq \|v\|_{\mathcal S_D} \leq C_K|v|_{\mathrm{coord}}$
for every $[f]\in K$ and $v\in T_{[f]}\mathcal S_D$.
\end{lemma}

\begin{proof}
The marked height--twist chart trivializes the tangent bundle over a
neighborhood of $K$, and $|\cdot|_{\mathrm{coord}}$ is a continuous
norm in this trivialization. By Lemma \ref{lem:split-entropy-norm},
$\|\cdot\|_{\mathcal S_D}$ is another continuous norm. Hence the
ratio
$\frac{\|v\|_{\mathcal S_D}}{|v|_{\mathrm{coord}}}$
is a positive continuous function on the compact coordinate unit
sphere bundle over $K$. Its minimum and maximum give the required
constants.
\end{proof}

\subsection{Closed stretching simplices}
\label{subsec:closed-stretching-families}

We recall the closed stretching construction of DeMarco--Pilgrim
\cite[Section 7.4]{DeMarco11}. It extends pure-height deformation
continuously to the non-generic locus.

Let $\mathbf h\in\mathcal H_D$ be generic and let $Q_0$ be a
connected component of $\mathcal G^{-1}(\mathbf h)$. By Proposition
\ref{prop:fiber_2.4}, $Q_0$ is a compact twist torus. Let
$\mathcal C_{Q_0}$ be the connected component of $\mathcal S_D^g$
containing $Q_0$. By \cite[Theorems 1.7 and 1.8]{DeMarco11}, its
image under
$\hat\pi\colon\mathcal S_D\to\mathbb P\mathcal{ST}_D^*$ is a
top-dimensional open simplex, denoted by $\sigma_{Q_0}$.

The projective critical heights map is affine and injective on
$\sigma_{Q_0}$, so we identify a point of $\sigma_{Q_0}$ with its
normalized critical height vector $x=(1,x_2,\ldots,x_{D-1})$. Let
$\overline\sigma_{Q_0}$ denote the corresponding closed simplex.

Fix $q\in Q_0$. Independent stretching gives, for
$x\in\sigma_{Q_0}$, a unique point
$\operatorname{Str}(q,x)\in\mathcal C_{Q_0}$ with the same largest
critical height as $q$ and projective critical height vector $x$.
The closed stretching construction extends this map continuously and
injectively to
$
\overline{\operatorname{Str}}_q\colon
\overline\sigma_{Q_0}\to \mathcal S_D;
$
see \cite[Section 7.4]{DeMarco11}.

For the proof of joint continuity below, we also fix the marked
critical B\"ottcher notation. Let $\mathcal P_D^\times$ be the space
of monic centered polynomials with marked critical points $c_i(f)$,
$1\leq i\leq D-1$, and let $\mathcal S_D^\times$ be its marked
shift locus. 
Let $p\colon\mathcal S_D^\times\to\mathcal S_D$ be the forgetful map.
We use the same notation for functions pulled back by $p$ and, on the
generic locus, for the pulled-back tangent norms. Set $M(f):=\max_iG_f(c_i(f))$. 
For $n\geq1$, let
$\mathcal S_D^\times(n)$ consist of the polynomials for which
$D^nG_f(c_i(f))>M(f)$ for every $i$, and set
$$
\mathcal W_D^\times(n):=
\left\{w\in(\mathbb C\setminus\overline{\mathbb D})^{D-1}:
D^n\log|w_i|>\max_j\log|w_j|\text{ for every }i\right\};
$$
see \cite[Section 2.3]{DeMarco11}. Define
$$
\Phi_n\colon\mathcal S_D^\times(n)\to\mathcal W_D^\times(n),
\qquad
\Phi_n(f)
:=
\bigl(
\varphi_f(f^n(c_1(f))),\ldots,
\varphi_f(f^n(c_{D-1}(f)))
\bigr),
$$
where $\varphi_f$ is the normalized B\"ottcher coordinate. The sets
$\mathcal S_D^\times(n)$ exhaust the marked shift locus, and by
\cite[Section 2.3]{DeMarco11}, $\Phi_n$ is holomorphic, proper,
open, and surjective; in particular, it has finite fibers and
$\log|\Phi_{n,i}(f)|=D^nG_f(c_i(f))$.

We shall use the following joint continuity property of the map $\overline{\operatorname{Str}}$.

\begin{lemma}
\label{lem:joint-closed-stretching}
The map
$
\overline{\operatorname{Str}}\colon
Q_0\times\overline\sigma_{Q_0}\to \mathcal S_D$ given by $
(q,x)\mapsto\overline{\operatorname{Str}}_q(x),
$
is continuous.
\end{lemma}

\begin{proof}
On $Q_0\times\sigma_{Q_0}$ the map is jointly continuous by the local
height--twist product structure. We need to control convergence
when the stretching parameter approaches a face of
$\overline\sigma_{Q_0}$.

We argue sequentially. Suppose $q_k\to q$ in $Q_0$ and $x_k\to x$ in
$\overline\sigma_{Q_0}$. Fix $x_*\in\sigma_{Q_0}$ and choose
$\varepsilon_m\downarrow0$. Set
$x^{(m)}:=(1-\varepsilon_m)x+\varepsilon_m x_* \in\sigma_{Q_0}.$
For each fixed $m$, interior joint continuity gives
$\operatorname{Str}(q_k,x^{(m)})\to
\operatorname{Str}(q,x^{(m)})$, while fixed-base continuity gives
$\operatorname{Str}(q,x^{(m)})\to
\overline{\operatorname{Str}}(q,x)$. Hence a diagonal choice yields
$m(k)\to\infty$ such that, with $y_k=x^{(m(k))}$, we have
\begin{equation}
\label{eq:joint-stretching-anchor}
\operatorname{Str}(q_k,y_k)
\to 
\overline{\operatorname{Str}}(q,x).
\end{equation}

Let $I_k$ be the segment joining $x_k$ to $y_k$ and set
$C_k:=\overline{\operatorname{Str}}_{q_k}(I_k).$
Each $C_k$ is compact and connected by the fixed-base closed
stretching construction. Let $\mathcal H_0$ be the set of
critical height vectors with largest height equal to that of $Q_0$
and projective vector in $\overline\sigma_{Q_0}$. The set $\mathcal H_0$ is a compact
subset of the positive cone: for one, hence every, fixed base point
$q\in Q_0$, the closed stretching map
$\overline{\operatorname{Str}}_q$ takes
$\overline\sigma_{Q_0}$ into the shift locus, and the critical
heights of the image points depend only on the stretching parameter. Therefore
$K:=\mathcal G^{-1}(\mathcal H_0)$
is a compact subset of $\mathcal S_D$ by properness of
$\mathcal G$. It contains $Q_0$ and every $C_k$. In particular, all
critical heights on $K$ are uniformly bounded above and bounded away
from zero.

Let $p\colon\mathcal S_D^\times\to\mathcal S_D$ be the finite map
from the critically marked shift locus, including the finite residual
ambiguity among monic centered representatives. The full preimage
$\widetilde K=p^{-1}(K)$ is compact. Since all critical heights on
$K$ are bounded above and bounded away from zero, there is a single
$n$ such that
$\widetilde K\subset\mathcal S_D^\times(n)$. In particular, the same
$n$ works simultaneously for the marked lifts of $C_k$, $Q_0$, and
the base points $q_k,q$. On this set define the phase map
$$
\Theta_n(\widetilde f)
:=
\left(
\frac{\Phi_{n,1}(\widetilde f)}
{|\Phi_{n,1}(\widetilde f)|},\ldots,
\frac{\Phi_{n,D-1}(\widetilde f)}
{|\Phi_{n,D-1}(\widetilde f)|}
\right).
$$
It is continuous. For each $i$, the phase of
$\Phi_{n,i}(\widetilde f)$ is precisely the external angle of the
point $f^n(c_i)$ in the B\"ottcher coordinate. Closed stretching
changes the relevant annular heights while preserving external
angles, so $\Theta_n$ is constant along a closed stretching segment
for a fixed critical marking; see
\cite[Sections 2.3 and 7.4]{DeMarco11}. Changing the critical marking
or the monic centered representative changes this vector only by a
finite group of permutations and residual rotations.

Pass to a subsequence for which $C_k$ converges in the Hausdorff
topology to a 
non-empty compact connected set $C\subset K$. We claim
that $C$ is finite. Let $f\in C$ and choose $f_k\in C_k$ with
$f_k\to f$. Choose marked lifts $\widetilde f_k\in\widetilde K$ and,
after passing to a subsequence, assume
$\widetilde f_k\to\widetilde f$. Choose marked lifts
$\widetilde q_k$ of $q_k$. Since the marked projection is finite and
proper, after passing to a further subsequence we may assume
$\widetilde q_k\to\widetilde q$ for some marked lift of $q$. The
remaining critical-marking and residual-rotation ambiguities form a
finite group. Stretching invariance and continuity of $\Theta_n$
therefore show that $\Theta_n(\widetilde f)$ belongs to the finite
orbit of $\Theta_n(\widetilde q)$ under this group.

The moduli critical height vector of $f$ is the one determined by
$x$. Therefore, after marking, the absolute values of the coordinates
of $\Phi_n(\widetilde f)$ are one of finitely many permutations of
$\left(e^{D^nh_1(q)},e^{D^nh_1(q)x_2},\ldots,
e^{D^nh_1(q)x_{D-1}}\right).$
Together with the preceding finite set of phases, this shows that
$\Phi_n(\widetilde f)$ belongs to a finite set $Z\subset
\mathcal W_D^\times(n)$. By \cite[Section 2.3]{DeMarco11}, $\Phi_n$ is a proper holomorphic
map between domains in $\mathbb C^{D-1}$. Hence each fiber is a
compact complex-analytic subset of such a domain. A positive-dimensional
compact analytic subset cannot occur in a domain of Euclidean space,
so every fiber is zero-dimensional and therefore finite. Thus
$\Phi_n^{-1}(Z)$ is finite. It follows that the set of possible
marked lifts $\widetilde f$, and hence the set $C$, is finite.

Since $C$ is connected, it is a single point. The anchor points
$\operatorname{Str}(q_k,y_k)$ belong to $C_k$, and
\eqref{eq:joint-stretching-anchor} shows that this point is
$\overline{\operatorname{Str}}(q,x)$. Consequently every Hausdorff
limit of a subsequence of $C_k$ is the singleton
$\{\overline{\operatorname{Str}}(q,x)\}$. In particular,
$\overline{\operatorname{Str}}(q_k,x_k)\to
\overline{\operatorname{Str}}(q,x)$, proving joint continuity.
\end{proof}

\subsection{Admissible curves and the path metric}
\label{subsec:admissible-curves}

The split norm is defined on the generic locus. Closed stretching
provides the distinguished pure-height curves by which we allow a
path to reach a non-generic point.

\begin{definition}
\label{def:adm-curves}
A continuous curve
$\gamma\colon[a,b]\to\mathcal S_D$ is \emph{admissible} if there is
a finite subdivision $a=t_0<\cdots<t_m=b$ such that each restriction
$\gamma|_{[t_{j-1},t_j]}$ is of one of the following types:
\begin{enumerate}
\item its image is contained in $\mathcal S_D^g$ and it is piecewise
$C^1$;
\item 
$\gamma((t_{j-1},t_j))\subset\mathcal S_D^g$, and
$\gamma|_{(t_{j-1},t_j)}$ is a smooth pure-height curve.
\end{enumerate}
Pieces of the second type may have non-generic endpoints. Constant
curves are admissible and have length zero. A piece of the second type
with a non-generic endpoint will be called a \emph{pure-height access
curve} to that endpoint.
\end{definition}

For an admissible curve, define
$L_D(\gamma) := \sum_j\int_{t_{j-1}}^{t_j} \|\dot\gamma(t)\|_{\mathcal S_D}\,dt,$
with an improper integral at a non-generic endpoint. Set
$\rho_D([f_0],[f_1]) := \inf_\gamma L_D(\gamma),$
where the infimum is taken over admissible curves joining
$[f_0]$ to $[f_1]$.

The following shows that this formulation gives exactly the same
path metric as piecewise pure-height and pure-twist motion.

\begin{proposition}
\label{prop:split-pure-equivalence}
The path function $\rho_D$ is unchanged if the infimum is restricted
to curves consisting entirely of finitely many pure-height and
pure-twist pieces with generic interiors.
\end{proposition}

\begin{proof}
Every such pure-piece curve is admissible, so it remains to
approximate a piecewise $C^1$ curve contained in the generic locus.

It is enough to consider a $C^1$ curve in one relatively compact
height--twist chart
$\Psi(B_g\times B_\theta)$, with $B_g$ and $B_\theta$ convex. Write
$\gamma(t)=\Psi(g(t),\theta(t))$. For a sufficiently fine partition
$a=t_0<\cdots<t_N=b$, replace the subarc from $t_i$ to $t_{i+1}$ by
the staircase
$(g(t_i),\theta(t_i)) \to (g(t_{i+1}),\theta(t_i)) \to (g(t_{i+1}),\theta(t_{i+1})).$
The first leg is pure height and the second is pure twist.

The two component norms depend continuously on the base point.
Uniform continuity on a compact neighborhood of the image of
$\gamma$ therefore shows, by the usual Riemann-sum argument, that the
lengths of these staircases converge to
$L_D(\gamma)$ as the mesh tends to zero. Applying this construction
to the finitely many generic pieces of an admissible curve gives, for
every $\varepsilon>0$, a pure-piece curve with the same endpoints and
length at most $L_D(\gamma)+\varepsilon$. Taking infima proves the
claim.
\end{proof}

We next show that finite-length admissible curves exist.

\begin{lemma}
\label{lem:finite-length-access}

Let $\mathcal U$ be a connected component of $\mathcal S_D^g$, and
let $p\in\overline{\mathcal U}^{\,\mathcal S_D}$. Let
$\sigma_{\mathcal U}\subset\mathbb P\mathcal{ST}_D^*$ be the
corresponding top-dimensional stretching simplex, and let
$x_0$ be the projective critical height vector of $p$. Fix
$x_*\in\sigma_{\mathcal U}$, and let $Q_*$ be the twist component in
$\mathcal U$ with projective critical height vector $x_*$ and largest
critical height $h_1(p)$. Then there is $q\in Q_*$ such that
$p=\overline{\operatorname{Str}}(q,x_0).$
Consequently, for every $x\in\sigma_{\mathcal U}$ the path
$t\mapsto \overline{\operatorname{Str}} \bigl(q,(1-t)x_0+tx\bigr), 0\leq t\leq1,$
has largest critical height $h_1(p)$, has generic interior in
$\mathcal U$, is pure height on its interior, and has finite
$\rho_D$-length. 
In particular, every point of
$\overline{\mathcal U}^{\,\mathcal S_D}$ admits finite-length pure-height access from
$\mathcal U$.
\end{lemma}

\begin{proof}
Choose $f_k\in\mathcal U$ with $f_k\to p$. After applying global
stretching factors converging to $1$, assume
$h_1(f_k)=h_1(p)$. Independently stretch $f_k$ to projective
critical height vector $x_*$ while keeping the largest critical
height fixed, and denote the resulting point by $g_k$. The points
$g_k$ lie in the fixed compact twist torus $Q_*$. Passing to a
subsequence, let $g_k\to q\in Q_*$. Lemma
\ref{lem:joint-closed-stretching} gives
$p=\overline{\operatorname{Str}}(q,x_0)$.

For $x\in\sigma_{\mathcal U}$, a segment from $x_0$ to $x$ lies in
$\sigma_{\mathcal U}$ except possibly at the initial endpoint. Hence
the displayed path has generic interior in $\mathcal U$ and is pure
height there. Its base rose path is affine and lies in a compact
subset of $\mathfrak M_{\mathfrak R}$. Smoothness of entropy
normalization and of the entropy metric on this compact set gives a
uniform bound for its entropy speed, so the path has finite length.
\end{proof}

The generic locus is dense in $\mathcal S_D$. Indeed, in a marked
neighborhood contained in some $\mathcal S_D^\times(n)$, the map
$\Phi_n$ is open, while points of its target whose radial coordinates
avoid the finitely many relations $g_i=D^ng_j$ are dense.

\begin{lemma}
\label{lem:finite-admissible-connectivity}
Any two points of $\mathcal S_D$ can be joined by an admissible curve
of finite length.
\end{lemma}

\begin{proof}
Let $p_0,p_1\in\mathcal S_D$. By
\cite[Corollary 6.2]{DeMarco11B}, choose a continuous path joining them,
and let $K$ be its compact connected image. By local finiteness, only finitely many top-dimensional simplices of
$\mathbb P\mathcal{ST}_D^*$ have closures meeting the compact set
$\hat\pi(K)$. Hence only finitely many generic components
$\mathcal U_1,\ldots,\mathcal U_N$ have closures meeting $K$.

Since the generic locus is dense, the closed sets
$F_j:=K\cap\overline{\mathcal U_j}, 1\leq j\leq N,$
cover $K$. Because $K$ is connected and the cover is finite, the
intersection graph of the non-empty sets $F_j$ is connected. Choose a
chain $F_{i_1},\ldots,F_{i_m}$ with $p_0\in F_{i_1}$,
$p_1\in F_{i_m}$, and
$F_{i_\nu}\cap F_{i_{\nu+1}}\neq
\emptyset$. Choose
$z_\nu\in F_{i_\nu}\cap F_{i_{\nu+1}}$ for $1\leq\nu<m$, and set
$z_0=p_0$, $z_m=p_1$.

For each $1\leq\nu\leq m$, both $z_{\nu-1}$ and $z_\nu$ lie in
$\overline{\mathcal U_{i_\nu}}$. By Lemma
\ref{lem:finite-length-access}, join each of them by a finite-length
pure-height access curve to a generic point of
$\mathcal U_{i_\nu}$. The two resulting generic points can be joined
inside the connected open manifold $\mathcal U_{i_\nu}$ by a
piecewise smooth generic path. Covering its compact image by finitely
many height--twist charts and applying Lemma
\ref{lem:generic-local-comparison} shows that it has finite length. Concatenating over the
finite chain gives a finite-length admissible curve from $p_0$ to
$p_1$.
\end{proof}

\subsection{Critical B\"ottcher coordinates and local separation}
\label{subsec:critical-bottcher}

It remains to prove that the path function separates points. We use
critical B\"ottcher coordinates, which record both the critical
heights and the angular information changed by twisting.

Recall the marked critical B\"ottcher-coordinate maps
$\Phi_n\colon\mathcal S_D^\times(n)\to\mathcal W_D^\times(n)$

introduced in Section \ref{subsec:closed-stretching-families}. They record both the critical heights and the
angular information changed by twisting. We now establish the local
Lipschitz estimate needed for separation and then descend to the
unmarked moduli space.

\begin{lemma}
\label{lem:elementary-bottcher-control}
Let $\widetilde U\Subset\mathcal S_D^\times(n)$. There is
$C_{\widetilde U}>0$ such that
$
|d\Phi_n(v)|
\leq
C_{\widetilde U}\|v\|_{\mathcal S_D}
$
for every generic $f\in\widetilde U$ and
$v\in T_f\mathcal S_D^\times$.
\end{lemma}

\begin{proof}
Write $v=v_H+v_T$ according to the height--twist splitting. The base
length functions $\ell_B(f)$ range in a compact subset of
$\mathfrak M_{\mathfrak R}$. On each of the two fixed coordinate
subspaces supported, respectively, on the first and on the last
$D-1$ rose coordinates, $dQ_{\ell_B(f)}$ is injective: the kernel of
$dQ$ is the radial line, and no non-zero vector in either subspace is
radial. Positive definiteness and compactness therefore give uniform
comparability between the Euclidean norm on either subspace and the
corresponding entropy norm after applying $dQ$.

We first consider the height part. On each marked generic chamber,
independent stretching preserves the external-angle coordinates, while
$\log|\Phi_{n,i}(f)| = D^nG_f(c_i(f)).$
Since $\widetilde U\Subset\mathcal S_D^\times(n)$, all critical
heights on $\overline{\widetilde U}$ are bounded above and bounded
away from zero. Hence, on each generic chamber meeting
$\widetilde U$, the change of variables from the ordered critical
heights to
$
(h_1,h_2/h_1,\ldots,h_{D-1}/h_1)
$
has uniformly bounded differential and inverse differential. The preceding uniform comparison therefore gives
$|d\Phi_n(v_H)| \leq C\|v_H\|_H$
with one constant on the generic part of $\widetilde U$.

We next consider the twist part. For a generic $f$ and
$1\leq j\leq D-1$, let
$$
Y_j(f)
:=
\left.\frac{d}{ds}\right|_{s=1}
(i,\ldots,i,is,i,\ldots,i)\cdot f
$$
be the current-point independent-stretching vector in the $j$-th
fundamental subannulus. By construction of independent wringing,
varying $s$ multiplies the modulus of the $j$-th fundamental
subannulus by $s$ and leaves the other fundamental moduli fixed.
Consequently,
$
dm_k(Y_j)=\delta_{jk}m_j.
$

On a fixed generic chamber, the fundamental moduli are linear
functions of suitable iterates of the critical heights, with an
invertible coefficient matrix; see the proof of
\cite[Lemma 5.3, equation (5.2)]{DeMarco11}. Thus the critical
heights are linear functions of the fundamental moduli on that
chamber. Since $\overline{\widetilde U}$ is compact in the shift
locus, only finitely many walls $g_a=D^qg_b$ occur there, and hence
only finitely many such linear systems occur. Therefore there is a
constant $C$ such that
$|dg_i(Y_j)| \leq C m_j$
for every generic $f\in\widetilde U$ and all $i,j$.

We now use holomorphicity to transfer this estimate from stretching
to twisting. Fix $f,i,j$ and restrict the $j$-th wringing orbit to a
small disk about $i\in\mathbb H$, writing $\tau=t+is$ and
$\tau\cdot_j f$ for wringing in the $j$-th factor. The orbit map
$\tau\mapsto\tau\cdot_j f$ is holomorphic by
\cite[Lemma 5.2]{DeMarco11}. Since $\Phi_{n,i}$ is holomorphic and
non-zero, choose a local holomorphic branch
$L_{i,j}(\tau) = \log\Phi_{n,i}(\tau\cdot_j f).$
Its real part is
$
D^nG_{\tau\cdot_j f}(c_i(\tau\cdot_j f)).
$
The imaginary-axis deformation is pure independent stretching and
preserves the corresponding external angle, so
$\partial_sL_{i,j}(i) = D^n\,dg_i(Y_j).$
The real-axis derivative is
$\partial_tL_{i,j}(i)=d(\log\Phi_{n,i})_f[X_j(f)]$.
The Cauchy--Riemann equations therefore give
$|\partial_tL_{i,j}(i)| = |\partial_sL_{i,j}(i)| \leq C m_j$.
Because $\Phi_n$ is bounded on $\overline{\widetilde U}$, we obtain
$|d\Phi_n(X_j)| \leq C m_j.$
Hence, if
$
v_T=\sum_j\xi_jX_j(f),
$
then
$|d\Phi_n(v_T)| \leq C\sum_{j=1}^{D-1}m_j|\xi_j|.$

It remains to compare the right-hand side with the twist norm.
The compact set
$(\hat\pi\circ p)(\overline{\widetilde U})$ meets only finitely many
generic top-dimensional simplices of
$\mathbb P\mathcal{ST}_D^*$. Within each corresponding generic component, independent wringing
preserves the topological covering degrees of the iterated preimage
annuli. Hence the integers $d_j$ are constant on that component. Thus the
integers $d_j$ are uniformly bounded on the generic part of
$\widetilde U$.

Set $a=(a_1,\ldots,a_{D-1})$, where
$
a_j
:=
\frac{\xi_j}
{H(f)\max\{1,d_j/m_j\}}.
$
Then
$m_j|\xi_j| = H(f)\max\{m_j,d_j\}|a_j|.$
On $\overline{\widetilde U}$, the function $H$ is bounded. The
moduli $m_j$ are also uniformly bounded, since
$
\sum_jm_j=(D-1)M(f)/(2\pi),
$
and the integers $d_j$ are uniformly bounded by the preceding
paragraph. It follows that
$
|d\Phi_n(v_T)|\leq C|a|.
$

The uniform comparison recorded at the beginning of the proof gives
$$
|a|
\leq
C
\left\|
dQ_{\ell_B(f)}
(0,\ldots,0,a_1,\ldots,a_{D-1})
\right\|_{\mathfrak h,\mathfrak R}
=
C\|v_T\|_T.
$$
Thus
$
|d\Phi_n(v_T)|
\leq
C\|v_T\|_T.
$
Combining the height and twist estimates yields
$$
|d\Phi_n(v)|
\leq
C_{\widetilde U}
\bigl(\|v_H\|_H+\|v_T\|_T\bigr)
=
C_{\widetilde U}\|v\|_{\mathcal S_D},
$$
as desired.
\end{proof}

We now descend to the unmarked moduli space. Let $\Gamma$ be the
finite group generated by changes of critical marking and residual
rotational conjugacies. Set
$\mathcal Y_n:=\mathcal W_D^\times(n)/\Gamma$, with quotient metric
$d_{\mathcal Y_n}([z],[w])=\min_{\sigma\in\Gamma}|z-\sigma w|$.
If $p\colon\mathcal S_D^\times(n)\to\mathcal S_D$ is the forgetful
map and $\mathcal U_n:=p(\mathcal S_D^\times(n))$, equivariance gives
a continuous proper finite-to-one map
$\overline\Phi_n\colon\mathcal U_n\to\mathcal Y_n$.

\begin{proposition}
\label{prop:local-bottcher-control}
Let $[f_0]\in\mathcal S_D$. There are $n\geq1$, a relatively compact
neighborhood $U\Subset\mathcal U_n$ of $[f_0]$, and $C_U>0$ such
that every admissible curve $\gamma\colon[a,b]\to U$ satisfies
$d_{\mathcal Y_n} \bigl( \overline\Phi_n(\gamma(a)), \overline\Phi_n(\gamma(b)) \bigr) \leq C_UL_D(\gamma).$
\end{proposition}

\begin{proof}
Choose $n$ so that a critically marked lift of $f_0$ lies in
$\mathcal S_D^\times(n)$. After shrinking around $[f_0]$, choose
$U\Subset\mathcal U_n$ so that all marked lifts of $U$ lie in a
fixed relatively compact set
$\widetilde U\Subset\mathcal S_D^\times(n)$.

Let $\gamma$ be admissible. Subdivide each generic piece further, if
necessary, so that every restriction $\gamma|_I$ to a subinterval
$I$ admits a lift $\widetilde\gamma_I$ to $\widetilde U$. Lemma
\ref{lem:elementary-bottcher-control} gives
$\operatorname{length}(\Phi_n\circ\widetilde\gamma_I)
\leq C_{\widetilde U}L_D(\gamma|_I)$.
Changes of marked lift differ by elements of the finite group $\Gamma$,
so the same bound holds for the length of
$\overline\Phi_n\circ\gamma|_I$ in the quotient.

For a pure-height piece with a non-generic endpoint, apply the
estimate on compact subintervals of the generic parameter interval
and let the corresponding subinterval endpoints tend to the parameter
value of the non-generic endpoint. Continuity of $\Phi_n$ gives
the same bound for the improper length.

Summing over the finite admissible subdivision proves
$
d_{\mathcal Y_n}
\bigl(
\overline\Phi_n(\gamma(a)),
\overline\Phi_n(\gamma(b))
\bigr)
\leq
C_UL_D(\gamma).
$
\end{proof}

\begin{corollary}
\label{cor:local-lower-separation}
Let $[f_0]\in\mathcal S_D$, and let $V$ be a neighborhood of
$[f_0]$. There is $\varepsilon>0$ such that every admissible curve
beginning at $[f_0]$ and leaving $V$ has length at least
$\varepsilon$.
\end{corollary}

\begin{proof}
Choose $\overline\Phi_n\colon U\to\mathcal Y_n$ as in Proposition
\ref{prop:local-bottcher-control}, with $[f_0]\in U\subset V$.
Since $\overline\Phi_n$ has finite fibers, choose an open
$U_0$ with $[f_0]\in U_0$, $\overline U_0\Subset U$, and
$\overline\Phi_n^{-1} \bigl(\overline\Phi_n([f_0])\bigr) \cap\overline U_0 = \{[f_0]\}.$
Hence
$
r
:=
\operatorname{dist}_{\mathcal Y_n}
\left(
\overline\Phi_n([f_0]),
\overline\Phi_n(\partial U_0)
\right)
>0.
$

Let $\gamma$ begin at $[f_0]$ and leave $V$, and let $t_*$ be its
first exit time from $U_0$. Then
$\gamma([a,t_*])\subset\overline U_0\subset U$ and
$\gamma(t_*)\in\partial U_0$. Proposition
\ref{prop:local-bottcher-control} gives
$r\leq C_UL_D(\gamma|_{[a,t_*]})\leq C_UL_D(\gamma)$.
Thus one may take $\varepsilon=r/C_U$.
\end{proof}

\subsection{The metric}
\label{subsec:shift-locus-metric}

\begin{proof}[Proof of Proposition \ref{thm:metric}]
By Lemma \ref{lem:finite-admissible-connectivity}, $\rho_D$ is finite.
Non-negativity and symmetry are immediate, and the triangle inequality
follows by concatenation. Proposition
\ref{prop:split-pure-equivalence} gives the equivalent description by
piecewise pure-height and pure-twist curves.

Let $[f_1]\neq[f_2]$, and choose a neighborhood $V$ of $[f_1]$ not
containing $[f_2]$. Corollary
\ref{cor:local-lower-separation} gives $\varepsilon>0$ such that every
admissible curve beginning at $[f_1]$ and leaving $V$ has length at
least $\varepsilon$. Every admissible curve from $[f_1]$ to $[f_2]$
leaves $V$, so
$\rho_D([f_1],[f_2])\geq\varepsilon$. Hence $\rho_D$ is a distance
function.
\end{proof}

\begin{corollary}
\label{cor:rhoD-dominates-rose-distance}
For all $[f_1],[f_2]\in\mathcal S_D$, we have
$
\rho_D([f_1],[f_2])
\geq
\rho_{\mathfrak h,\mathfrak R}
\bigl(
\hat\ell_B([f_1]),\hat\ell_B([f_2])
\bigr).
$
\end{corollary}

\begin{proof}
Let $\gamma$ be admissible. At every generic point write
$\dot\gamma=v_H+v_T$. Since twisting preserves the critical heights,
$d\hat\ell_B(v_T)=0$. Therefore
$$
\left\|
\frac{d}{dt}\hat\ell_B(\gamma(t))
\right\|_{\mathfrak h,\mathfrak R}
=
\|v_H\|_H
\leq
\|\dot\gamma(t)\|_{\mathcal S_D}.
$$
The same inequality holds on a pure-height access piece. Integrating
and taking the infimum over admissible curves proves the result.
\end{proof}

\begin{proposition}
\label{prop:metric-topology}
The metric $\rho_D$ induces the usual topology on $\mathcal S_D$.
\end{proposition}

\begin{proof}
We first show that convergence in $\rho_D$ implies convergence in
the usual topology. This follows directly from Corollary
\ref{cor:local-lower-separation}.

We prove the converse by reducing to generic points with a fixed
largest critical height. Let $p_k\to p$ in the usual topology and
set $s:=h_1(p)>0$. By Lemma \ref{lem:finite-length-access} and
smoothness of $Q$ and of the entropy metric on compact positive
sets, each non-generic $p_k$ can be joined to a generic point
$\widetilde p_k$ of the same largest critical height by an access
curve of length less than $1/k$. Choose the endpoints sufficiently
close to $p_k$ that $\widetilde p_k\to p$; if $p_k$ is generic,
take $\widetilde p_k=p_k$. Globally stretching $\widetilde p_k$
by $s/h_1(\widetilde p_k)$ preserves convergence to $p$ and changes
the points by $\rho_D$-distance tending to zero. Indeed, the
corresponding generic radial segments have base length functions
in a fixed compact positive set, and their first-coordinate
variations tend to zero. By the triangle inequality, it therefore
suffices to treat generic $p_k$ with $h_1(p_k)=s$.

We next choose a single stretching family in which to compare the
sequence with its limit. Suppose, on the contrary, that after
passing to a subsequence we have $\rho_D(p_k,p)\geq\varepsilon$
for some $\varepsilon>0$. By local finiteness of
$\mathbb P\mathcal{ST}_D^*$ and continuity of $\hat\pi$, only
finitely many generic components meet a sufficiently small
neighborhood of $p$. Passing to a further subsequence, assume
that all $p_k$ lie in one component $\mathcal U$. Fix a reference
twist torus $Q_*$ in $\mathcal U$ at largest critical height $s$,
and write $p_k=\operatorname{Str}(q_k,x_k)$ with $q_k\in Q_*$.
Compactness gives, after a further subsequence, $q_k\to q\in Q_*$,
while $x_k\to x_0$, the projective critical height vector of $p$.
Lemma \ref{lem:joint-closed-stretching} gives
$p=\overline{\operatorname{Str}}(q,x_0)$.
Set $y_k:=\operatorname{Str}(q,x_k)$. We will connect $p_k$ to
$y_k$ by a short twist path and $y_k$ to $p$ by a short height path.

We show that the twist paths from $p_k$ to $y_k$ have lengths
tending to zero. Choose local reference twist coordinates near
$q$ and join $q_k$ to $q$ by a straight segment with displacement
$\Delta\theta_k\to0$. Stretch this segment to projective critical
height vector $x_k$; the resulting pure-twist path joins $p_k$
to $y_k$. Let $m_j^*$ be the reference fundamental moduli and
$s_{j,k}>0$ the stretching factors, so the stretched moduli are
$m_j=s_{j,k}m_j^*$. Along the stretched path, the current-point
speeds satisfy $\xi_j=\dot\theta_j/s_{j,k}$, and hence
$$
\frac{\xi_j}{H\max\{1,d_j/m_j\}}
=
\frac{m_j^*\dot\theta_j}{H\max\{m_j,d_j\}}.
$$
Since $d_j\geq1$ and $H\geq s$, the auxiliary rose tangent is
bounded by a constant times $|\dot\theta|$, uniformly in $k$,
even if some stretched moduli tend to zero. The base rose length
functions remain in a fixed compact positive set, so smoothness
of $Q$ and of the entropy metric gives
$\rho_D(p_k,y_k)\leq C|\Delta\theta_k|\to0$.

We finally show that the height paths from $y_k$ to $p$ have
lengths tending to zero. The closed stretching path
$t\mapsto\overline{\operatorname{Str}}
(q,(1-t)x_k+tx_0)$ has generic interior and is pure height.
Its base rose path is affine, lies in a fixed compact positive
set, and has Euclidean length $|x_k-x_0|$. Smoothness of $Q$
and of the entropy metric therefore gives
$\rho_D(y_k,p)\leq C|x_k-x_0|\to0$.
Combining the two path estimates yields
$\rho_D(p_k,p)\leq C|\Delta\theta_k|+C|x_k-x_0|\to0$,
contradicting $\rho_D(p_k,p)\geq\varepsilon$.
This proves the converse.
\end{proof}

\section{Behavior of the entropy}
\label{sec:entropy-behavior}
In this section, we collect the entropy estimates used in the study of
the metric completion. We first describe the entropy normalization of
a divergent family of rose metrics in terms of its smallest length
scale in Section \ref{subsec:entropy-minimal-scale}. We then specialize this description to the base length
functions arising from the shift locus in Section \ref{subsec:entropy-base-length-functions}. Finally, in Section \ref{subsec:coercivity} we prove two
large-scale estimates: a coercive lower bound in the projective
critical height directions and a uniform upper bound for certain
non-singular monotone height paths.

\subsection{Entropy and the minimal scale}
\label{subsec:entropy-minimal-scale}

Throughout this section, let $\mathfrak R=\mathfrak R_n, n \ge 2$ and write
$\mathfrak h(\ell):=\mathfrak h_{\mathfrak R}(\ell)$.

\begin{definition}
\label{def:divergent-length-functions}
Let $\{\ell_k\}\subset\mathfrak M_{\mathfrak R}$.

\begin{enumerate}
\item
The sequence is \emph{divergent} if it is not relatively compact in
$\mathfrak M_{\mathfrak R}$.

\item
It is \emph{uniformly divergent} if every coordinate
$\ell_k(e_j)$ converges in $[0,+\infty]$ and at least one coordinate
has limit $0$ or $+\infty$.
\end{enumerate}
\end{definition}

Every divergent sequence has a uniformly divergent subsequence.
Indeed, choose a subsequence leaving every compact subset of
$\mathbb R_{>0}^n$. Since $[0,+\infty]^n$ is compact, a further
subsequence converges coordinatewise in $[0,+\infty]^n$. Since the
subsequence leaves every compact subset of $\mathbb R_{>0}^n$, at
least one limiting coordinate is $0$ or $+\infty$. After
passing to a further subsequence, we may also assume that every ratio
$\ell_k(e_i)/\ell_k(e_j)$ converges in $[0,+\infty]$. We call such a
subsequence \emph{scale-regular}.

\begin{definition}
\label{def:minimal-scale-set}
Let $\{\ell_k\}$ be scale-regular and uniformly divergent. If every
coordinate tends to $+\infty$, set
${\bf I}(\{\ell_k\})=\emptyset$. Otherwise, the
\emph{minimal-scale set} ${\bf I}={\bf I}(\{\ell_k\})$ is the unique
non-empty subset of $\{1,\ldots,n\}$ such that
$$
\ell_k(e_i)\asymp\ell_k(e_{i'})
\quad (i,i'\in{\bf I}),
\qquad
\ell_k(e_i)=o(\ell_k(e_j))
\quad (i\in{\bf I},\ j\notin{\bf I}).
$$
\end{definition}

The passage to a scale-regular subsequence will be understood whenever
the minimal-scale set is used.

\begin{proposition}
\label{prop:entropy-minimal-scale}
Let $\{\ell_k\}\subset\mathfrak M_{\mathfrak R}$ be scale-regular and
uniformly divergent, set
$\mathfrak h_k:=\mathfrak h(\ell_k)$, and let
${\bf I}={\bf I}(\{\ell_k\})$.

\begin{enumerate}
\item
If ${\bf I}=\emptyset$, then $\mathfrak h_k\to0$.

\item
If ${\bf I}=\{i\}$, then
$\mathfrak h_k\ell_k(e_i)\to0$, while
$\mathfrak h_k\ell_k(e_j)\to+\infty$ for every $j\neq i$.

\item
If $|{\bf I}|\geq2$, then
$\mathfrak h_k\asymp1/\ell_k(e_i)$ for every $i\in{\bf I}$.
\end{enumerate}
\end{proposition}

\begin{proof}
We use the entropy equation
\begin{equation}
\label{eq:minimal-scale-entropy}
1=
\sum_{\emptyset\neq S\subset\{1,\ldots,n\}}
(2|S|-1)e^{-\mathfrak h_k\ell_k(S)}.
\end{equation}

If ${\bf I}=\emptyset$, every coordinate tends to $+\infty$. If
$\mathfrak h_k$ were bounded below along a subsequence, every term on
the right-hand side of
\eqref{eq:minimal-scale-entropy} would tend to zero, a contradiction.
This proves (1).

Suppose that ${\bf I}=\{i\}$ and set
$a_k:=\mathfrak h_k\ell_k(e_i)$. If $a_k$ were bounded below by a
positive constant along a subsequence, then
$\mathfrak h_k\ell_k(e_j)\to+\infty$ for every $j\neq i$. The
right-hand side of \eqref{eq:minimal-scale-entropy} would then
converge to $e^{-a}$ for some $a>0$, or to zero, and could not equal
one. Thus $a_k\to0$.

If $\mathfrak h_k\ell_k(e_j)$ remained bounded for some $j\neq i$,
the singleton term associated to $i$ would tend to $1$, while the
singleton term associated to $j$ would remain bounded below by a
positive constant. This again contradicts
\eqref{eq:minimal-scale-entropy}, proving (2).

Finally, suppose that $|{\bf I}|\geq2$. Fix $i_0\in{\bf I}$ and set
$a_k:=\mathfrak h_k\ell_k(e_{i_0})$. If $a_k\to0$ along a subsequence,
the singleton terms corresponding to two distinct indices in
${\bf I}$ would both tend to $1$. If $a_k\to+\infty$, every term in
\eqref{eq:minimal-scale-entropy} would tend to zero. Hence $a_k$
remains bounded above and below by positive constants. Since all
coordinates indexed by ${\bf I}$ are comparable, (3) follows.
\end{proof}

\begin{corollary}
\label{cor:normalized-minimal-scale}
Under the assumptions of Proposition
\ref{prop:entropy-minimal-scale}, set
$\hat\ell_k:=\mathfrak h_k\ell_k$.

\begin{enumerate}
\item
If ${\bf I}=\{i\}$, then
$\hat\ell_k(e_i)\to0$ and
$\hat\ell_k(e_j)\to+\infty$ for every $j\neq i$.

\item
If $|{\bf I}|\geq2$, then
$\hat\ell_k(e_i)$, $i\in{\bf I}$, remain bounded above and below by
positive constants, while
$\hat\ell_k(e_j)\to+\infty$ for $j\notin{\bf I}$.
\end{enumerate}
\end{corollary}

\begin{proof}
This is immediate from Proposition
\ref{prop:entropy-minimal-scale}.
\end{proof}

\subsection{Base length functions}
\label{subsec:entropy-base-length-functions}
For the remainder of this section, set $n=2D-2$ and
$\mathfrak R=\mathfrak R_{2D-2}$.
We now specialize to the base length functions of Section
\ref{sec_3}. For $\{[f_k]\}\subset\mathcal S_D$, write
$\ell_k:=\ell_B([f_k])$ and
$\hat\ell_k:=\hat\ell_B([f_k])$. Thus
$$
\ell_k
=
\left(
h_1(f_k),
\frac{h_2(f_k)}{h_1(f_k)},\ldots,
\frac{h_{D-1}(f_k)}{h_1(f_k)},
1,\ldots,1
\right).
$$

\begin{definition}
\label{def:divergent-shift-sequence}
A sequence $\{[f_k]\}\subset\mathcal S_D$ is called
\emph{divergent} if it is not relatively compact in $\mathcal S_D$.
It is called \emph{uniformly divergent} or \emph{scale-regular} if
its base length functions have the corresponding property. We call
the sequence \emph{degenerating} if $h_1(f_k)\to+\infty$.
\end{definition}

A sequence in $\mathcal S_D$ is divergent if and only if its base
length functions are not relatively compact. Indeed, relative
compactness of $\{\ell_k\}$ gives upper and lower bounds for
$h_1(f_k)$ and for all ratios $h_j(f_k)/h_1(f_k)$. Hence the critical
height vectors remain in a compact subset of the positive cone in
$\mathcal H_D$, and properness of $\mathcal G$ gives relative
compactness in $\mathcal S_D$.

Equivalently, after passing to a subsequence, divergence occurs in one
of the following ways: $h_1(f_k)\to0$, $h_1(f_k)\to+\infty$, or
$h_{D-1}(f_k)\to0$ while $h_1(f_k)$ remains bounded above and below.

Since the last $D-1$ coordinates of every base length function equal
$1$, its minimal-scale set is never empty.

\begin{definition}
\label{def:singular-base-lengths}
Let $\{[f_k]\}\subset\mathcal S_D$ be uniformly divergent and
scale-regular. We call the sequence \emph{singular} if the
minimal-scale set of its base length functions has cardinality one,
and \emph{non-singular} otherwise.
\end{definition}

\subsection{Large-height estimates}
\label{subsec:coercivity}
Throughout this subsection, assume $D\geq3$. We next study base rose metrics of the form
$\ell_{s,r} = (s,r_2,\ldots,r_{D-1},1,\ldots,1),$
where $s$ is the largest critical height and
$r=(r_2,\ldots,r_{D-1})$ is the projective critical height vector.

The first estimate shows that the projective-height directions do not
collapse when $s\to+\infty$.

\begin{lemma}
\label{lem:asymptotic-projective-coercivity}
Let $K\Subset\mathbb R_{>0}^{D-2}$. There is $c_K>0$ such that, for
$s\geq1$, $r\in K$, $u\in\mathbb R$, and
$w\in\mathbb R^{D-2}$,
$$
\left\|
dQ_{\ell_{s,r}}
(u,w,0,\ldots,0)
\right\|_{\mathfrak h,\mathfrak R}^2
\geq
c_K|w|^2.
$$
\end{lemma}

\begin{proof}
Write $h=\mathfrak h_{\mathfrak R}(\ell_{s,r})$ and
$L=Q(\ell_{s,r})=h\ell_{s,r}$. Since $D\geq3$, at least two auxiliary rose edges have length
$1$. The two corresponding singleton terms in the entropy equation
give $1\geq2e^{-h}$, hence $h\geq\log2$. On the other hand, all
coordinates of $\ell_{s,r}$ are bounded below by a positive constant
depending only on $K$; since the entropy equation is a finite positive
sum, this gives a uniform upper bound for $h$. Thus there are
constants $0<a_K\leq b_K<\infty$ such that
$a_K\leq h\leq b_K$
uniformly for $s\geq1$ and $r\in K$.

Let $v=(u,w,0,\ldots,0)$ and set
$\eta:=D\mathfrak h_{\mathfrak R}|_{\ell_{s,r}}[v]$. Then
$V:=dQ_{\ell_{s,r}}v = (hu+s\eta,\;hw+r\eta,\;\eta,\ldots,\eta),$
and $V$ is tangent to the unit-entropy hypersurface at $L$.
Apply the positive formula
\eqref{eq:rose-entropy-norm-positive-formula}. Its denominator is
uniformly bounded above, since there are only finitely many subsets
and $xe^{-x}\leq e^{-1}$ for $x>0$. On the other hand, all coordinates
of $L$ except possibly the first are uniformly bounded above. Keeping
only the singleton terms in the numerator corresponding to those
coordinates gives
$$
\|V\|_{\mathfrak h,\mathfrak R}^2
\geq
c_K\bigl(|hw+r\eta|^2+\eta^2\bigr).
$$
For fixed $w$, minimizing the expression on the right over $\eta$
gives
$$
|hw+r\eta|^2+\eta^2
\geq
\frac{h^2}{1+|r|^2}|w|^2.
$$
Since $h$ is bounded below and $r$ ranges in the compact set $K$, the
claim follows.
\end{proof}

We also record the complementary upper bound that we will need.
The bound is expressed in terms of the unnormalized input vector and
controls all coordinates of its image under $dQ$, including the
potentially large first coordinate.

\begin{lemma}
\label{lem:asymptotic-transverse-upper-bound}
Let $K\Subset\mathbb R_{>0}^{D-2}$. There is $C_K>0$ such that, for
$s\geq1$, $r\in K$, and every vector $w\in\mathbb R^{2D-3}$ supported
on all rose coordinates except the first, we have
$\left\| dQ_{\ell_{s,r}}(0,w) \right\|_{\mathfrak h,\mathfrak R} \leq C_K|w|.$
\end{lemma}

\begin{proof}
Write $h=\mathfrak h_{\mathfrak R}(\ell_{s,r})$ and
$L=h\ell_{s,r}$. By the proof of Lemma \ref{lem:asymptotic-projective-coercivity}, $h$ is uniformly bounded
above and below for $s\geq1$ and $r\in K$. Let $v=(0,w)$ be supported
away from the first rose coordinate and set
$\eta:=D\mathfrak h_{\mathfrak R}|_{\ell_{s,r}}[v]$.
Differentiating the entropy equation once gives
$$
\eta
=
-h\,
\frac{
\displaystyle\sum_{\emptyset\neq S}
(2|S|-1)e^{-h\ell_{s,r}(S)}v(S)
}{
\displaystyle\sum_{\emptyset\neq S}
(2|S|-1)e^{-h\ell_{s,r}(S)}\ell_{s,r}(S)
}.
$$
The denominator contains an auxiliary singleton term and is uniformly
bounded below, while the numerator is a finite sum of terms linear in
$w$. Hence $|\eta|\leq C_K|w|$.

Now
$V:=dQ_{\ell_{s,r}}v=h v+\eta\ell_{s,r}.$
In the positive formula
\eqref{eq:rose-entropy-norm-positive-formula}, the denominator has a
uniform positive lower bound from any auxiliary singleton term. If $1\notin S$, then
$V(S)=O_K(|w|)$. If $1\in S$, then
$V(S)$ may contain the term $s\eta$, but the weight $e^{-L(S)}$
contains the factor $e^{-hs}$; since $h$ is uniformly bounded below,
$s^2e^{-hs}$ is uniformly bounded. Thus the numerator is at most
$C_K|w|^2$, and the stated upper bound follows.
\end{proof}

The next estimate gives a uniform upper bound for monotone
projective-height paths along which the two smallest critical heights
remain comparable. It is convenient to formulate it for
a slightly larger class of monotone paths.

For the statement, set $r_1(t)\equiv1$.

\begin{lemma}
\label{lem:nonsingular-diagonal-path}
Fix $0<c\leq1$. There is $C_c>0$ with the following property. Let
$s\geq1$, and let
$\ell(t) = \bigl( s,r_2(t),\ldots,r_{D-1}(t),1,\ldots,1 \bigr), 0\leq t\leq1,$
be a piecewise $C^1$ path satisfying
$1=r_1(t)\geq r_2(t)\geq\cdots\geq r_{D-1}(t)>0$
and
$r_{D-1}(t)\geq c\,r_{D-2}(t)$ for all $t$.

Set $a(t):=r_{D-1}(t)$ and
$z(t):=a(t)^{-1}\ell(t)$. Assume that every coordinate of $z(t)$ is
monotone. Then
$L_{\mathfrak h,\mathfrak R}(Q\circ\ell)\leq C_c.$
The constant is independent of $s$ and of the path.

In particular, the conclusion applies to the affine path
$r_j(t)=(1-t)r_j(0)+t$ whenever
$r_{D-1}(0)\geq c\,r_{D-2}(0)$, and to the pushing-up path
$r_j(t)=\max\{r_j(0),t\}$, parametrized by
$r_{D-1}(0)\leq t\leq1$.
\end{lemma}

\begin{proof}
Since entropy normalization is invariant under scaling,
$Q(\ell(t))=Q(z(t))$. Every coordinate of $z(t)$ is at least $1$,
and the coordinate corresponding to $r_{D-1}(t)$ is identically
equal to $1$.

There is also a second coordinate which is uniformly bounded above.
If $D\geq4$, then
$1\leq \frac{r_{D-2}(t)}{r_{D-1}(t)} \leq c^{-1}.$
If $D=3$, the hypothesis reads $r_2(t)\geq c$, and each of the last
two rose coordinates of $z(t)$ is
$1/r_2(t)\leq c^{-1}$.
Thus, in every degree $D\geq3$, at least two coordinates of $z(t)$
remain in the fixed compact interval $[1,c^{-1}]$.

Write $h(t):=\mathfrak h_{\mathfrak R}(z(t))$ and
$\Phi(t):=Q(z(t))=h(t)z(t)$. The entropy equation shows that there are
constants $0<A_c<B_c<+\infty$ such that
$A_c\leq h(t)\leq B_c$ for every path satisfying the hypotheses of
the lemma. Indeed, two
singleton coordinates remain bounded above by $c^{-1}$, which gives
the lower bound for $h$, while all coordinates of $z$ are at least
one, which gives the upper bound.

We first estimate $\dot h$. Differentiating the entropy equation gives
$$
\frac{\dot h}{h}
=
-
\frac{
\sum_{\emptyset\neq S}
(2|S|-1)\dot z(S)e^{-hz(S)}
}{
\sum_{\emptyset\neq S}
(2|S|-1)z(S)e^{-hz(S)}
}.
$$
The denominator is bounded below by a positive constant, since it
contains the singleton term associated to $z_{D-1}=1$. Since
$h\geq A_c$ and the number of subsets is finite, there are constants
$C,\alpha>0$, depending only on $c$ and $D$, such that
\begin{equation}
\label{eq:monotone-h-derivative}
|\dot h(t)|
\leq
C\sum_{j=1}^{2D-2}
e^{-\alpha z_j(t)}|\dot z_j(t)|.
\end{equation}

We now apply the positive formula
\eqref{eq:rose-entropy-norm-positive-formula} to the unit-entropy length function $\Phi(t)\in\mathfrak M_{\mathfrak R}^1$ and its tangent vector $\dot\Phi(t)$. Since the
coordinate $z_{D-1}$ is identically $1$, the singleton term
$S=\{D-1\}$ gives a uniform positive lower bound for the denominator:
$\sum_{\emptyset\neq S} (2|S|-1)e^{-\Phi(S)}\Phi(S) \geq h(t)e^{-h(t)} \geq A_c e^{-B_c}.$
For every 
non-empty $S$, we have
$\dot\Phi(S)=h\dot z(S)+\dot h\,z(S)$. Using
\eqref{eq:monotone-h-derivative}, the bounds on $h$, and the boundedness
of $u e^{-A_cu/2}$ for $u\geq0$, we absorb the factor $z(S)$ in
$z(S)e^{-\Phi(S)/2}$ and obtain, after decreasing $\alpha$ if necessary,
$|\dot\Phi(S)|e^{-\Phi(S)/2} \leq C\sum_{j=1}^{2D-2} e^{-\alpha z_j(t)}|\dot z_j(t)|.$
There are only finitely many subsets $S$, so
\eqref{eq:rose-entropy-norm-positive-formula} gives
$$
\|\dot\Phi(t)\|_{\mathfrak h,\mathfrak R}
\leq
C\sum_{j=1}^{2D-2}
e^{-\alpha z_j(t)}|\dot z_j(t)|.
$$

Each $z_j$ is monotone and takes values in $[1,+\infty)$. Therefore
$\int_0^1 e^{-\alpha z_j(t)}|\dot z_j(t)|\,dt \leq \int_1^\infty e^{-\alpha u}\,du.$
Summing over the finitely many coordinates gives
$L_{\mathfrak h,\mathfrak R}(Q\circ\ell)\leq C_c$.

For the affine path, division by
$a(t)=(1-t)r_{D-1}(0)+t$ shows directly that every coordinate of
$z(t)$ is monotone. The same is true for the pushing-up path, where
$a(t)=t$ and
$r_j(t)/a(t)=\max\{r_j(0)/t,1\}$. This proves the final assertion.
\end{proof}

\section{The metric completion}
\label{sec:metric-completion}

In this section, we prove Theorem \ref{thm:completion}. In Section
\ref{subsec:infinite-distance}, we show that singular rose
degenerations lie at infinite $\rho_D$-distance and deduce
completeness in degree two. Section
\ref{subsec:short-positive-connectors} establishes the large-height
estimates for projective-height, twist, and radial motion. In Section \ref{subsec:monotone-collapse}, we prove that sequences
converging to the same point of $\mathbb P\mathcal{ST}_D^*$ become
asymptotically indistinguishable, while Section
\ref{subsec:projective-model-separation} gives the complementary
separation estimate. Finally, Section
\ref{subsec:positive-height-boundary} identifies the positive-height
metric boundary with $\mathbb P\mathcal{ST}_D^*$ and completes the
proof.

Throughout the section, write
$\mathfrak R:=\mathfrak R_{2D-2}$ and, for $[f]\in\mathcal S_D$, set
$\ell_f:=\ell_B([f])$ and $\hat\ell_f:=\hat\ell_B([f])$.

\subsection{Infinite distance and the quadratic case}
\label{subsec:infinite-distance}
We begin with the singular degenerations detected in Section
\ref{sec:entropy-behavior}.

\begin{lemma}
\label{lem:singular-infinite-distance}
Let $\{[f_k]\}\subset\mathcal S_D$ be uniformly divergent and
scale-regular. If the minimal-scale set of $\{\ell_{f_k}\}$ has
cardinality one, then, for every fixed $k_0$, we have
$\rho_D([f_{k_0}],[f_k])\to+\infty$.
\end{lemma}

\begin{proof}
Write ${\bf I}(\{\ell_{f_k}\})=\{i_0\}$. By Corollary
\ref{cor:normalized-minimal-scale},
$\hat\ell_{f_k}(e_{i_0})\to0$, while
$\hat\ell_{f_k}(e_j)\to+\infty$ for $j\neq i_0$. In particular,
$\min_j\hat\ell_{f_k}(e_j)\to0$.

Lemma \ref{lem:rose-infinite-diameter} implies that the
$\rho_{\mathfrak h,\mathfrak R}$-distance from $\hat\ell_{f_k}$ to
any fixed point tends to infinity. The conclusion follows from
Corollary \ref{cor:rhoD-dominates-rose-distance}.
\end{proof}

The criterion in Lemma \ref{lem:singular-infinite-distance} has the following concrete consequence for
large height degenerations in the shift locus.

\begin{corollary}
\label{cor:singular-degenerations}
Let $\{[f_k]\}\subset\mathcal S_D$ satisfy
$h_1(f_k)\to+\infty$. If $D=2$, then
$\rho_D([f_{k_0}],[f_k])\to+\infty$ for every fixed $k_0$.

If $D\geq3$ and
$h_{D-1}(f_k)/h_{D-2}(f_k)\to0$, then the same conclusion holds.
\end{corollary}

\begin{proof}
For $D=2$, the base length functions are
$\ell_{f_k}=(h_1(f_k),1)$, so their minimal-scale sets are
singletons. The conclusion follows from Lemma \ref{lem:singular-infinite-distance}.

Now suppose that $D\geq3$ and
$h_{D-1}(f_k)/h_{D-2}(f_k)\to0$. If the conclusion failed, there
would be a subsequence whose distances from $[f_{k_0}]$ remain
bounded. Passing to a scale-regular further subsequence,
$h_{D-1}(f_k)/h_1(f_k)$ is asymptotically smaller than every other
coordinate of $\ell_{f_k}$. Its minimal-scale set is therefore a
singleton, contradicting Lemma
\ref{lem:singular-infinite-distance}.
\end{proof}

In particular, $(\mathcal S_D,\rho_D)$ has infinite diameter for
every $D\geq2$.

\begin{proposition}
\label{prop:quadratic-complete}
The metric space $(\mathcal S_2,\rho_2)$ is complete.
\end{proposition}

\begin{proof}
Let $\{[f_k]\}$ be a $\rho_2$-Cauchy sequence. Since
$\ell_{f_k}=(h_1(f_k),1)$, Lemma
\ref{lem:singular-infinite-distance} shows that $h_1(f_k)$ can tend
neither to zero nor to infinity along a subsequence. Hence the
critical heights remain in a compact subset of $\mathbb R_{>0}$.

Properness of $\mathcal G$ places $\{[f_k]\}$ in a compact subset of
$\mathcal S_2$. Passing to a subsequence, suppose that
$[f_k]\to[f_\infty]$ in the usual topology. Every point of
$\mathcal S_2$ is generic, and Lemma
\ref{lem:generic-local-comparison} shows that, in a sufficiently small
height--twist chart around $[f_\infty]$, the $\rho_2$-distance is
bounded above by a constant multiple of the coordinate distance.
Thus $[f_k]\to[f_\infty]$ in $\rho_2$ along the subsequence. Since the
original sequence is Cauchy, the full sequence converges to
$[f_\infty]$.
\end{proof}

For completeness, the quadratic level-set behavior in Figure \ref{fig:1}
also follows directly from the twist normalization. Set
$\kappa(x):=\|dQ_{(x,1)}(0,1)\|_{\mathfrak h,\mathfrak R_2}$.
The entropy equation
$1=e^{-hx}+e^{-h}+3e^{-h(x+1)}$, with
$h=\mathfrak h_{\mathfrak R_2}(x,1)$, and
\eqref{eq:rose-entropy-norm-positive-formula} give
$\kappa(x)\asymp|\log x|^{-1/2}$ as $x\to0$ and as $x\to+\infty$.
Since $m_1=x/(2\pi)$ and the twist period $P(x):=d_1/m_1$
is comparable to $x^{-1}$, the length of the level set is at most
$\kappa(x)P(x)/(H(x)\max\{1,P(x)\})$, where
$H(x)=\max\{x,x^{-1}\}$; this upper bound tends to zero at both ends.

For the remainder of this section, assume $D\geq3$.

\subsection{Short motions at large height}
\label{subsec:short-positive-connectors}

We record two mechanisms by which wringing motion becomes inexpensive in the
positive-height regime. We show that projective-height motion is controlled
uniformly by the entropy metric, while twisting has cost of order
$h_1^{-1}$.

Write $r([f])=(r_2(f),\ldots,r_{D-1}(f))\in
\mathbb R_{>0}^{D-2}$, where $r_j(f):=h_j(f)/h_1(f)$; the full
normalized critical height vector is $(1,r([f]))$.

\begin{lemma}
\label{lem:large-height-slice-control}
Let $K\Subset\mathbb R_{>0}^{D-2}$. There is $C_K>0$ such that the
following holds. Let $[f]\in\mathcal S_D^g$ satisfy
$h_1(f)=s\geq1$ and
$\left( \frac{h_2(f)}{h_1(f)},\ldots, \frac{h_{D-1}(f)}{h_1(f)} \right)\in K.$
If $v_H\in H_{[f]}$ is tangent to the slice $h_1=s$, then we have
$
\|v_H\|_H\leq C_K\,|dr(v_H)|.
$
\end{lemma}

\begin{proof}
The base tangent corresponding to $v_H$ has the form
$(0,dr(v_H),0,\ldots,0),$
because $v_H$ is tangent to the slice $h_1=s$.  This vector is
supported away from the first rose coordinate.  Lemma
\ref{lem:asymptotic-transverse-upper-bound} therefore gives
$
\|v_H\|_H
=
\left\|
dQ_{\ell_{s,r}}(0,dr(v_H),0,\ldots,0)
\right\|_{\mathfrak h,\mathfrak R}
\leq C_K|dr(v_H)|,
$
as required.
\end{proof}

We next control motion inside a generic critical height fiber.

\begin{lemma}
\label{lem:twist-components-shrink}
Let $\{[f_k]\}\subset\mathcal S_D^g$ satisfy
$h_1(f_k)\to+\infty$, and suppose that their projective critical
heights remain in a fixed compact subset of the positive locus. Let
$T_k$ be the twist-deformation component of the critical height fiber
containing $[f_k]$. Then we have
$
\operatorname{diam}_{\rho_D}(T_k)\to0.
$
More precisely,
$\operatorname{diam}_{\rho_D}(T_k)=O(h_1(f_k)^{-1})$.
\end{lemma}

\begin{proof}
Pass to a marked lift of $T_k$ and use the current-point unit twist
vectors $X_1,\ldots,X_{D-1}$. Twisting preserves the critical heights
and the fundamental moduli, so $m_j$ and $d_j$ are constant along the
twist orbit. By \eqref{eq_dj/mj}, the current-point twist path
$t\mapsto(i,\ldots,i,t+i,i,\ldots,i)\cdot f$ in the $j$-th
factor returns modulo the stabilizer after parameter time $d_j/m_j$. Hence any two points of $T_k$
can be joined by a concatenation of coordinate twist arcs for which
the current-point speeds satisfy
$
\int |\xi_j(t)|\,dt\leq d_j/m_j
$
for each $j$.

Along this path the critical heights are constant. 
Writing $\gamma$ for this path, the integral of the absolute value
of the $j$-th auxiliary rose tangent component therefore satisfies
$$
\int\left|\delta\ell_{\gamma(t)}(\dot\gamma(t))(e_{D-1+j})\right|\,dt
\leq\frac{d_j/m_j}{H([f_k])\max\{1,d_j/m_j\}}
\leq\frac{1}{H([f_k])}.
$$
Along a twist path the base rose metric is fixed, and the auxiliary
rose tangent is supported on the last $D-1$ coordinates.  By Lemma
\ref{lem:asymptotic-transverse-upper-bound}, uniformly over the fixed
compact set of projective critical heights,
$$
\left\|
dQ_{\ell_B([f_k])}(\delta\ell_{f_k}(v_T))
\right\|_{\mathfrak h,\mathfrak R}
\leq C\,|\delta\ell_{f_k}(v_T)|.
$$

Integrating and using the preceding integral estimate in
each of the $D-1$ twist coordinates gives a total length
$O(H([f_k])^{-1})$.  Since $H([f_k])\geq h_1(f_k)$, the claimed
estimate follows.
\end{proof}

Finally, we control the radial direction.

\begin{lemma}
\label{lem:large-stretch-short}
Let $K\Subset\mathbb R_{>0}^{D-2}$. There are constants
$C_K,c_K>0$ such that, for every $r\in K$ and $a,b\geq1$, the monotone
radial segment
$\ell(t)=\bigl(s(t),r,1,\ldots,1\bigr),$
with $s$ varying from $a$ to $b$, satisfies
$
L_{\mathfrak h,\mathfrak R}(Q\circ\ell)
\leq
C_K e^{-c_K\min\{a,b\}}.
$
\end{lemma}

\begin{proof}
Write $\ell_{s,r}=(s,r_2,\ldots,r_{D-1},1,\ldots,1)$,
$h=\mathfrak h_{\mathfrak R}(\ell_{s,r})$, and let $v=e_1$. 
Here $e_1=(1,0,\ldots,0)$ denotes the first coordinate vector. As in
Lemma \ref{lem:asymptotic-projective-coercivity}, for $r\in K$ there
are uniform constants $0<\alpha_K\leq h\leq\beta_K$. Set
$\eta:=D\mathfrak h_{\mathfrak R}|_{\ell_{s,r}}[e_1]$. Differentiating
the entropy equation once gives
$$
\eta
=
-h\,
\frac{
\displaystyle\sum_{S\ni1}
(2|S|-1)e^{-h\ell_{s,r}(S)}
}{
\displaystyle\sum_{\emptyset\neq S}
(2|S|-1)e^{-h\ell_{s,r}(S)}\ell_{s,r}(S)
}.
$$
Every term in the numerator contains the first edge and therefore a
factor $e^{-hs}$, while the denominator has a uniform positive lower
bound from an auxiliary singleton term. Hence
$|\eta|\leq Ce^{-\alpha_Ks}$.

Set $V:=dQ_{\ell_{s,r}}e_1=h e_1+\eta\ell_{s,r}$. Apply
\eqref{eq:rose-entropy-norm-positive-formula}. The denominator again
has a uniform positive lower bound. For a subset containing the first
edge, its contribution to the numerator is $O(e^{-\alpha_Ks})$, since
$V(S)=h+O(se^{-\alpha_Ks})$. For a subset not containing the first edge,
$V(S)=O(e^{-\alpha_Ks})$. Hence, after changing the constants,
$\|dQ_{\ell_{s,r}}e_1\|_{\mathfrak h,\mathfrak R} \leq Ce^{-cs}$
uniformly for $r\in K$. Integrating from $\min\{a,b\}$ to
$\max\{a,b\}$ gives
$$
L_{\mathfrak h,\mathfrak R}(Q\circ\ell)
\leq
\int_{\min\{a,b\}}^\infty Ce^{-cs}\,ds
\leq
C_K e^{-c_K\min\{a,b\}},
$$
after adjusting the constants.
\end{proof}

\subsection{Collapse over the monotone quotient}
\label{subsec:monotone-collapse}

The main result of this subsection is the following.

\begin{proposition}
\label{prop:asymptotic-model-collapse}
Suppose that $f_k,g_k\in\mathcal S_D$ satisfy
$\hat\pi(f_k)\to\zeta$ and $\hat\pi(g_k)\to\zeta$ for some
$\zeta\in\mathbb P\mathcal{ST}_D^*$, while
$h_1(f_k),h_1(g_k)\to+\infty$. Then
$\rho_D(f_k,g_k)\to0$.
\end{proposition}

We first prove the following local collapse lemma using one twist
correction and one closed stretching segment.

\begin{lemma}
\label{lem:local-sequential-collapse}
Let $p\in\mathcal S_D$ satisfy $h_1(p)=1$. Suppose that
$p_k,q_k\to p$, with $h_1(p_k)=h_1(q_k)=1$, and let
$s_k\to+\infty$. Then we have
$
\rho_D(s_k\cdot p_k,s_k\cdot q_k)\to0.
$
\end{lemma}

\begin{proof}
It is enough to prove
$\rho_D(s_k\cdot p_k,s_k\cdot p)\to0$; the same argument applies to
$q_k$, and the triangle inequality then gives the result.
Let $d_{\mathrm{top}}$ be a metric inducing the usual topology on
$\mathcal S_D$, and write $\operatorname{Var}_r(\beta)$ for the
Euclidean total variation of $r\circ\beta$.

First replace $p_k$, if necessary, by a nearby generic point. If
$p_k$ is non-generic, choose a generic component having $p_k$ in its
closure and use Lemma \ref{lem:finite-length-access}. Since the access
path is obtained from a straight segment in a closed stretching
simplex, we may choose a point $\widetilde p_k$ on the generic
interior of the access path so close to $p_k$ that
$$
d_{\mathrm{top}}(p_k,\widetilde p_k)<k^{-1},
\qquad
\operatorname{Var}_r
\bigl(p_k\rightsquigarrow\widetilde p_k\bigr)<k^{-1}.
$$
If $p_k$ is generic, set $\widetilde p_k:=p_k$. Thus
$\widetilde p_k\to p$ and $h_1(\widetilde p_k)=1$. For all sufficiently large $k$, the projective-height images of these
access paths lie in a fixed compact neighborhood of $r(p)$ in the
positive locus. Lemma \ref{lem:large-height-slice-control} therefore gives
$\rho_D(s_k\cdot p_k,s_k\cdot\widetilde p_k) \leq Ck^{-1}$.

Only finitely many generic components meet a sufficiently small
neighborhood of $p$. Indeed, by local finiteness of
$\mathbb P\mathcal{ST}_D^*$, a neighborhood of $\hat\pi(p)$ meets
only finitely many top-dimensional simplices; by continuity of
$\hat\pi$, the same is true for a sufficiently small neighborhood of
$p$. Hence, if the desired conclusion failed, we could pass to a
subsequence for which all $\widetilde p_k$ lie in one fixed generic
component $\mathcal U$.
Since $\widetilde p_k\to p$, we have $p\in\overline{\mathcal U}$.
Apply Lemma \ref{lem:finite-length-access} and write
$S(x):=\overline{\operatorname{Str}}(q,x), S(1,r(p))=p.$
Set $y_k:=S(1,r(\widetilde p_k))$. The points
$\widetilde p_k$ and $y_k$ have the same critical heights and lie in
the same generic component. Hence they lie in the same twist
component of that generic critical height fiber. Lemma
\ref{lem:twist-components-shrink} gives
$\rho_D(s_k\cdot\widetilde p_k,s_k\cdot y_k)\to0.$

Finally, join $y_k$ to $p$ by the closed-stretching segment
$t\mapsto
S\bigl(1,(1-t)r(\widetilde p_k)+tr(p)\bigr).$
Its generic interior is pure height, its largest critical height is
$1$, and its total projective-height variation is
$|r(\widetilde p_k)-r(p)|$. After stretching by $s_k$, Lemma
\ref{lem:large-height-slice-control} therefore gives
$\rho_D(s_k\cdot y_k,s_k\cdot p) \leq C|r(\widetilde p_k)-r(p)| \to0.$
Combining the three estimates proves
$\rho_D(s_k\cdot p_k,s_k\cdot p)\to0$ and hence the lemma.
\end{proof}

The same idea gives collapse for two points which are approached from
the same generic component.

\begin{lemma}
\label{lem:common-incident-component-collapse}
Let $\mathcal U$ be a connected component of $\mathcal S_D^g$, and
let $p,q\in\overline{\mathcal U}$ satisfy
$\mathcal G(p)=\mathcal G(q) =(1,h_2,\ldots,h_{D-1}), h_{D-1}>0.$
If $s_k\to+\infty$, then
$\rho_D(s_k\cdot p,s_k\cdot q)\to0.$
\end{lemma}

\begin{proof}
Let $x_0=(1,r(p))=(1,r(q))$. Fix an interior projective-height vector
$x_*\in\sigma_{\mathcal U}$ and let $Q_*$ be the corresponding twist
torus. Applying Lemma \ref{lem:finite-length-access} to $p$
and $q$, with the same $x_*$ and $Q_*$, gives $a,b\in Q_*$ such that
$p=\overline{\operatorname{Str}}(a,x_0), q=\overline{\operatorname{Str}}(b,x_0).$

For $0<\varepsilon<1$, set
$x_\varepsilon:=(1-\varepsilon)x_0+\varepsilon x_*$,
$p_\varepsilon:=\operatorname{Str}(a,x_\varepsilon)$, and
$q_\varepsilon:=\operatorname{Str}(b,x_\varepsilon)$.
The two points lie in the same generic twist component. Therefore
Lemma \ref{lem:twist-components-shrink} gives, for fixed
$\varepsilon$,
$\rho_D(s_k\cdot p_\varepsilon,s_k\cdot q_\varepsilon) \to0.$
On the other hand, the closed-stretching segments from $p$ to
$p_\varepsilon$ and from $q$ to $q_\varepsilon$ are pure height on
their generic interiors and have projective-height variation
$\varepsilon|x_*-x_0|$. Hence Lemma
\ref{lem:large-height-slice-control} yields
$\limsup_{k\to\infty} \rho_D(s_k\cdot p,s_k\cdot q) \leq 2C\varepsilon|x_*-x_0|.$
Letting $\varepsilon\to0$ proves the lemma.
\end{proof}

We now collapse an arbitrary connected component of a critical height
fiber. The point is that only finitely many generic components can be
incident to a fixed compact fiber component.

\begin{lemma}
\label{lem:component-sequential-collapse}
Let
$h=(1,h_2,\ldots,h_{D-1})$ have all coordinates positive, and let
$\mathcal C$ be a connected component of $\mathcal G^{-1}(h)$.
Suppose that $p_k\to p\in\mathcal C$ and
$q_k\to q\in\mathcal C$, with
$h_1(p_k)=h_1(q_k)=1$, and let $s_k\to+\infty$. Then we have
$
\rho_D(s_k\cdot p_k,s_k\cdot q_k)\to0.
$
\end{lemma}

\begin{proof}
Lemma \ref{lem:local-sequential-collapse} gives
$\rho_D(s_k\cdot p_k,s_k\cdot p)\to0$ and $\rho_D(s_k\cdot q,s_k\cdot q_k)\to0.$
It remains to compare $s_k\cdot p$ and $s_k\cdot q$.

The compact set $\hat\pi(\mathcal C)$ meets the closures of only
finitely many top-dimensional simplices of
$\mathbb P\mathcal{ST}_D^*$ by local finiteness. Consequently there
are only finitely many generic components
$\mathcal U_1,\ldots,\mathcal U_N$ whose closures meet
$\mathcal C$.

By the density of the generic locus established in Section \ref{subsec:admissible-curves}, the closed sets
$F_j:=\mathcal C\cap\overline{\mathcal U_j},  1\leq j\leq N,$
cover $\mathcal C$.

Since $\mathcal C$ is connected and the cover is finite, the
intersection graph of the non-empty sets $F_j$ is connected. Choose
a chain
$F_{i_1},\ldots,F_{i_m}$ with
$p\in F_{i_1}$, $q\in F_{i_m}$, and
$F_{i_\nu}\cap F_{i_{\nu+1}}\neq
\emptyset$. Choose
$z_\nu\in F_{i_\nu}\cap F_{i_{\nu+1}}$ for
$1\leq\nu<m$, and set $z_0:=p$, $z_m:=q$. For each
$1\leq\nu\leq m$, the pair $z_{\nu-1},z_\nu$ lies in the closure of
the same generic component $\mathcal U_{i_\nu}$ and has the same
critical height vector $h$. Lemma
\ref{lem:common-incident-component-collapse} therefore gives
$\rho_D(s_k\cdot z_{\nu-1},s_k\cdot z_\nu)\to0.$
There are only finitely many terms, so the triangle inequality gives
$\rho_D(s_k\cdot p,s_k\cdot q)\to0$. Combining this with the first two
estimates proves the lemma.
\end{proof}

The following lemma synchronizes the radial scales.

\begin{lemma}
\label{lem:radial-synchronization}
Let $K\Subset\mathbb R_{>0}^{D-2}$. There are constants
$C_K,c_K>0$ such that, whenever $p\in\mathcal S_D$ satisfies
$h_1(p)=1$ and $r(p)\in K$, then, for all $a,b\geq1$, we have 
$
\rho_D(a\cdot p,b\cdot p)
\leq
C_K e^{-c_K\min\{a,b\}}.
$
\end{lemma}

\begin{proof}
Choose a compact neighborhood $K'\Subset\mathbb R_{>0}^{D-2}$ of
$K$ in the positive locus. Let $\mathcal U$ be a generic component
whose closure contains $p$. By Lemma
\ref{lem:finite-length-access}, for every sufficiently small
$\varepsilon>0$ there is a generic point $p_\varepsilon\in\mathcal U$
with $h_1(p_\varepsilon)=1$, $r(p_\varepsilon)\in K'$, and a
closed-stretching access path from $p$ to $p_\varepsilon$ whose entire
projective-height image lies in $K'$ and whose total
projective-height variation is at most $\varepsilon$. This is possible
because the access path is a straight segment in projective-height
coordinates and $K'$ is a neighborhood of $K$.

After stretching this access path by $a$ and by $b$, Lemma
\ref{lem:large-height-slice-control} gives
$\rho_D(a\cdot p,a\cdot p_\varepsilon) + \rho_D(b\cdot p,b\cdot p_\varepsilon) \leq 2C_{K'}\varepsilon.$
The radial segment from $a\cdot p_\varepsilon$ to
$b\cdot p_\varepsilon$ is generic with constant projective critical
heights, so Lemma \ref{lem:large-stretch-short} gives
$\rho_D(a\cdot p_\varepsilon,b\cdot p_\varepsilon) \leq C_{K'}e^{-c_{K'}\min\{a,b\}}.$
The triangle inequality therefore gives the same estimate with an
additional term $2C_{K'}\varepsilon$. Letting
$\varepsilon\to0$ proves the claim, after renaming the constants.
\end{proof}

\begin{proof}[Proof of Proposition \ref{prop:asymptotic-model-collapse}]
Set
$s_k=\min\{h_1(f_k),h_1(g_k)\}$ and normalize
$\bar f_k=h_1(f_k)^{-1}\cdot f_k, \bar g_k=h_1(g_k)^{-1}\cdot g_k.$
Since the projective critical height map is continuous and
$\hat\pi(f_k),\hat\pi(g_k)\to\zeta$, the projective critical heights
of both normalized sequences remain in a fixed compact subset of the
positive locus. Lemma \ref{lem:radial-synchronization} gives
$\rho_D(f_k,s_k\cdot\bar f_k)\to0 $ and $\rho_D(s_k\cdot\bar g_k,g_k)\to0.$

Suppose that
$\rho_D(s_k\cdot\bar f_k,s_k\cdot\bar g_k)$ does not tend to zero.
After passing to a subsequence, properness of $\mathcal G$ on the
normalized positive slice gives
$\bar f_k\to p$ and $\bar g_k\to q$.

Continuity and stretching invariance of $\hat\pi$ imply
$\hat\pi(p)=\hat\pi(q)=\zeta$. On the slice $h_1=1$, a fiber of
$\hat\pi$ is a connected component of a critical height fiber. Thus
$p$ and $q$ belong to the same connected component of a positive
critical height fiber. Lemma
\ref{lem:component-sequential-collapse} gives
$
\rho_D(s_k\cdot\bar f_k,s_k\cdot\bar g_k)\to0,
$
a contradiction.
\end{proof}

\subsection{Separation of projective models}
\label{subsec:projective-model-separation}

We now prove the complementary separation statement.

\begin{proposition}
\label{prop:positive-model-separation}
Let $\zeta\neq\eta$ be points of
$\mathbb P\mathcal{ST}_D^*$. There are disjoint relatively compact
neighborhoods $\mathcal V_\zeta,\mathcal V_\eta$ and
$\varepsilon>0$ such that
$\rho_D([f],[g])\geq\varepsilon$ whenever
$\hat\pi([f])\in\mathcal V_\zeta,
\hat\pi([g])\in\mathcal V_\eta,
$
and $h_1(f),h_1(g)$ are sufficiently large.
\end{proposition}

Let
$\pi\colon \mathbb P\mathcal{ST}_D^* \to \mathbb P\mathcal{SH}_D$
be the projective critical heights map. Equip the locally finite
simplicial complex $\mathbb P\mathcal{ST}_D^*$ with a piecewise
Euclidean metric by identifying each simplex with a regular Euclidean
simplex of edge length $1$, so the induced metrics agree on common
faces. Let
$d_*$ denote the corresponding intrinsic path metric. We write $L_*$ for the associated path length and $|v|_*$ for the
Euclidean norm of a tangent vector within a simplex.

\begin{lemma}
\label{lem:model-metric}
The function $d_*$ is a metric on
$\mathbb P\mathcal{ST}_D^*$ and induces its simplicial topology.
\end{lemma}

\begin{proof}
The space $\mathbb P\mathcal{ST}_D^*$ is connected: it is the
continuous image of the connected shift locus under the surjective map
$\hat\pi$. Hence its underlying simplicial complex is connected and,
in particular, any two points can be joined by a finite piecewise
linear path through simplices. Thus $d_*$ is finite everywhere.

Fix a point $p$ and a neighborhood $V$ of $p$. By local finiteness,
after subdividing near $p$, choose a finite closed star contained in
$V$ whose interior contains $p$. A continuous piecewise linear
function $\chi$, supported in this star and satisfying $\chi(p)=1$,
has uniformly bounded slopes, say by $C$. Every path from $p$ leaving
the star therefore has length at least $1/C$. Thus paths leaving the
finite star cannot produce arbitrarily short shortcuts. Within the
star, the finite-complex length topology agrees with the simplicial
topology. These two observations show that $d_*$ separates points and
induces the simplicial topology.
\end{proof}

\begin{lemma}
\label{lem:projective-model-estimate}
Let
$K\Subset\mathbb P\mathcal{SH}_D$ lie in the positive locus. There
is $C_K>0$ such that every admissible curve $\gamma$ satisfying
$h_1(\gamma(t))\geq1$ and $\pi(\hat\pi(\gamma(t)))\in K$ obeys
$L_*(\hat\pi\circ\gamma)\leq C_KL_D(\gamma).$
\end{lemma}

\begin{proof}
The set $\pi^{-1}(K)$ is compact and, by local finiteness, meets only
finitely many simplices. On each such simplex, $\pi$ is affine and
injective. Hence there is a uniform constant $A_K>0$ such that
$|v|_* \leq A_K|d\pi(v)|$
for every tangent vector in these simplices.

Consider a generic smooth piece of $\gamma$ and write
$\dot\gamma=v_H+v_T$. Twisting preserves the connected component of
the critical height fiber, so $d\hat\pi(v_T)=0$. Write
$s=h_1$ and
$r=(h_2/h_1,\ldots,h_{D-1}/h_1)$. The base rose tangent
corresponding to $v_H$ has the form
$(\dot s,\dot r,0,\ldots,0)$.

By Lemma \ref{lem:asymptotic-projective-coercivity},
$\|v_H\|_H\geq c_K|\dot r|$
uniformly for $h_1\geq1$ above $K$. Since
$d\pi(d\hat\pi(v_H))=\dot r$, we obtain
$$
|(\hat\pi\circ\gamma)'|_*
\leq
A_K|\dot r|
\leq
C_K\|v_H\|_H
\leq
C_K\|\dot\gamma\|_{\mathcal S_D}.
$$

If an admissible piece has a non-generic endpoint, it is pure height
on its generic interior, so the same estimate applies there and
extends by integration. Summing over the admissible subdivision proves
the lemma.
\end{proof}

\begin{proof}[Proof of Proposition \ref{prop:positive-model-separation}]
Choose $\mathcal V_\zeta$ and $\mathcal V_\eta$ with compact disjoint
closures. Set $\delta:=d_*(\overline{\mathcal V_\zeta},
\overline{\mathcal V_\eta})>0$.
After shrinking the neighborhoods, assume that their projective
critical height images lie in a compact positive set $K_0$. Choose a
compact positive set $K$ whose interior contains $K_0$, let $C_K$ be
given by Lemma \ref{lem:projective-model-estimate}, and set $R:=1$.

For $r\in K$, set
$\ell_{R,r}=(R,r_2,\ldots,r_{D-1},1,\ldots,1)$ and let
$A_R := \{Q(\ell_{R,r}):r\in K\}.$
For $r\in K_0$, let $L_\infty(r)$ denote the limiting normalized rose
metric obtained by sending the first edge to $+\infty$: its finite
coordinates are obtained by entropy normalizing the subrose with
length vector $(r_2,\ldots,r_{D-1},1,\ldots,1)$, and its first
coordinate is $+\infty$. Set
$A_\infty:=\{L_\infty(r):r\in K_0\}$.

The sets $A_R$ and $A_\infty$ are disjoint compact subsets of the
completed unit-entropy metric rose space. Let
$d_0:=\rho_{\mathfrak h,\mathfrak R}(A_R,A_\infty)>0.$
By Lemma \ref{lem:large-stretch-short}, letting the terminal radial
parameter tend to infinity gives constants $C,c>0$ such that
$$
\sup_{r\in K_0}
\rho_{\mathfrak h,\mathfrak R}
\bigl(Q(\ell_{s,r}),L_\infty(r)\bigr)
\leq Ce^{-cs},\qquad s\geq1,
$$
where $\ell_{s,r}=(s,r_2,\ldots,r_{D-1},1,\ldots,1)$.
For sufficiently large $R_0\geq R$, every point with
$h_1\geq R_0$ and projective critical heights in $K_0$ has normalized
base rose metric within $d_0/4$ of $A_\infty$.

Let $\gamma$ join a point above $\mathcal V_\zeta$ to a point above
$\mathcal V_\eta$, both with largest critical height at least $R_0$.
Stop the path at the first time that one of the following occurs:
\begin{enumerate}
\item
$\hat\pi\circ\gamma$ reaches
$\overline{\mathcal V_\eta}$;

\item
$\pi\circ\hat\pi\circ\gamma$ leaves
$\operatorname{int}K$;

\item
$h_1\circ\gamma$ reaches $R$.
\end{enumerate}

In the first case,
Lemma \ref{lem:projective-model-estimate} gives
$L_D(\gamma)\geq\delta/C_K$.

In the second case, the compact set
$\overline{\mathcal V_\zeta}$ and the closed set
$
\mathbb P\mathcal{ST}_D^*
\setminus
\pi^{-1}(\operatorname{int}K)
$
are disjoint. Their $d_*$-distance is therefore some
$\delta_K>0$, and
$L_D(\gamma)\geq\delta_K/C_K$.

In the third case, the initial normalized rose metric lies within
$d_0/4$ of $A_\infty$, while the normalized rose metric at the first
point with $h_1=R$ belongs to $A_R$. Corollary
\ref{cor:rhoD-dominates-rose-distance} gives
$L_D(\gamma)\geq d_0/2$.

Taking the minimum of these three positive bounds proves the
proposition.
\end{proof}

\subsection{The positive-height boundary}
\label{subsec:positive-height-boundary}

Assume from now on that $D\geq3$, and let
$\hat{\mathcal S}_D$ be the metric completion of
$(\mathcal S_D,\rho_D)$. Recall that
$\mathcal C_D^{\infty,+}$ consists of completion classes admitting
a $\rho_D$-Cauchy representative $\{[f_k]\}$ with
$h_1(f_k)\to+\infty$ and
$h_j(f_k)/h_1(f_k)\to a_j>0$ for every $2\leq j\leq D-1$.

We first check that this positive-height condition depends only on the
completion point.

\begin{lemma}
\label{lem:positive-regime-well-defined}
Let $\{[f_k]\}$ and $\{[g_k]\}$ be equivalent
$\rho_D$-Cauchy sequences. Suppose that
$h_1(f_k)\to+\infty$ and
$\frac{h_j(f_k)}{h_1(f_k)} \to a_j>0, 2\leq j\leq D-1.$
Then $h_1(g_k)\to+\infty$ and
$h_j(g_k)/h_1(g_k)\to a_j$ for $2\leq j\leq D-1$.
\end{lemma}

\begin{proof}
Let $a=(a_2,\ldots,a_{D-1})$, and let $L_\infty(a)$ denote the
limiting unit-entropy rose metric associated to
$(\infty,a_2,\ldots,a_{D-1},1,\ldots,1)$. Then
$\hat\ell_{f_k}\to L_\infty(a)$.

Since $\rho_D(f_k,g_k)\to0$, Corollary
\ref{cor:rhoD-dominates-rose-distance} gives
$
\rho_{\mathfrak h,\mathfrak R}
(\hat\ell_{f_k},\hat\ell_{g_k})\to0.
$
Hence $\hat\ell_{g_k}$ has the same limit. By
\cite[Proposition 9.6(3)]{Aougab23}, convergence in the completed
entropy metric is equivalent to coordinate convergence. The last
$D-1$ coordinates of $\hat\ell_{g_k}$ are equal to
$\mathfrak h_{\mathfrak R}(\ell_{g_k})$ and therefore converge to a
positive number. The first
coordinate forces $h_1(g_k)\to+\infty$, while the next $D-2$
coordinates give the required projective-height limits.
\end{proof}

Thus $\mathcal C_D^{\infty,+}$ is well-defined.

\begin{lemma}
\label{lem:positive-cauchy-model-limit}
Let $x\in\mathcal C_D^{\infty,+}$ be represented by
$\{[f_k]\}$. Then there is a unique
$\zeta\in\mathbb P\mathcal{ST}_D^*$ such that
$\hat\pi([f_k])\to\zeta$.
\end{lemma}

\begin{proof}
The projective critical heights converge in the positive locus. Choose
a compact positive neighborhood $K$ of their limit. Properness of
$\pi$ makes $\pi^{-1}(K)$ compact, so
$\{\hat\pi([f_k])\}$ has a convergent subsequence.

If two subsequences converged to distinct limits, Proposition
\ref{prop:positive-model-separation} would give a positive lower
bound for the distance between sufficiently far-out terms of the two
subsequences, contradicting the Cauchy property. Thus the
subsequential limit is unique, and compactness gives convergence of
the full sequence.
\end{proof}

\begin{proposition}
\label{prop:boundary-class}
For every $\zeta\in\mathbb P\mathcal{ST}_D^*$, there is a unique
class
$\Psi(\zeta)\in\mathcal C_D^{\infty,+}$ represented by a sequence
$\{[f_k]\}$ satisfying
$\hat\pi([f_k])\to\zeta$ and $h_1(f_k)\to+\infty$.
\end{proposition}

\begin{proof}
Choose $p_\zeta\in\mathcal S_D$ with
$h_1(p_\zeta)=1$ and $\hat\pi(p_\zeta)=\zeta$. For any
$s_k\to+\infty$, set $f_k:=s_k\cdot p_\zeta$. Lemma
\ref{lem:radial-synchronization} shows that $\{f_k\}$ is
$\rho_D$-Cauchy. It leaves every compact subset of $\mathcal S_D$ by
properness of $\mathcal G$. By Corollary
\ref{cor:local-lower-separation}, metric convergence to a point of the
shift locus would force ordinary convergence there. Hence $\{f_k\}$
does not converge in $\mathcal S_D$ and therefore defines a boundary
point in $\mathcal C_D^{\infty,+}$.

If $q_\zeta$ is another normalized point above $\zeta$, and
$t_k\to+\infty$, Proposition
\ref{prop:asymptotic-model-collapse} gives
$
\rho_D(s_k\cdot p_\zeta,t_k\cdot q_\zeta)\to0.
$
Thus the completion class is independent of all choices. We denote it
by $\Psi(\zeta)$.
For any sequence $\{f_k\}$ satisfying the conditions of the proposition,
set $g_k:=h_1(f_k)\cdot p_\zeta$. Proposition
\ref{prop:asymptotic-model-collapse} gives $\rho_D(f_k,g_k)\to0$,
while Lemma \ref{lem:radial-synchronization} shows that $\{g_k\}$ is
Cauchy and represents $\Psi(\zeta)$. Hence $\{f_k\}$ is also Cauchy
and represents the same class, proving uniqueness.
\end{proof}

\begin{proposition}
\label{prop:boundary-homeomorphism}
The map
$\Psi\colon
\mathbb P\mathcal{ST}_D^*
\to
\mathcal C_D^{\infty,+}$
is a homeomorphism.
\end{proposition}

\begin{proof}
Injectivity follows from Proposition
\ref{prop:positive-model-separation}.

For surjectivity, let $x\in\mathcal C_D^{\infty,+}$ be represented by
$\{[f_k]\}$. By Lemma
\ref{lem:positive-cauchy-model-limit}, there is
$\zeta\in\mathbb P\mathcal{ST}_D^*$ such that
$\hat\pi([f_k])\to\zeta$. Choose a normalized point $p_\zeta$ above
$\zeta$ and set $g_k:=h_1(f_k)\cdot p_\zeta$. Proposition
\ref{prop:asymptotic-model-collapse} gives
$\rho_D(f_k,g_k)\to0$. Thus $x=\Psi(\zeta)$.

We next prove continuity. Suppose that $\zeta_m\to\zeta$. Choose
normalized points $p_m,p$ above $\zeta_m,\zeta$, respectively. Choose
$s_m\to+\infty$ sufficiently large that
$s_m\cdot p_m$ and $s_m\cdot p$ lie within $1/m$ of
$\Psi(\zeta_m)$ and $\Psi(\zeta)$ in the metric completion.
Proposition \ref{prop:asymptotic-model-collapse} gives
$\rho_D(s_m\cdot p_m,s_m\cdot p)\to0.$
Hence $\Psi(\zeta_m)\to\Psi(\zeta)$.

Finally, suppose that
$\Psi(\zeta_m)\to\Psi(\zeta)$. Choose normalized $p_m,p$ and scales
$s_m\to+\infty$ as above. Then
$\rho_D(s_m\cdot p_m,s_m\cdot p)\to0$. Corollary
\ref{cor:rhoD-dominates-rose-distance} shows that both
$\hat\ell_B(s_m\cdot p_m)$ and $\hat\ell_B(s_m\cdot p)$ converge to
$L_\infty(r(p))$. By
\cite[Proposition 9.6 (3)]{Aougab23}, the completed entropy metric has
the coordinate topology. Hence their auxiliary coordinates converge
to the same positive entropy factor, and division by this factor gives
$\pi(\zeta_m)\to\pi(\zeta)$.

Choose a compact positive neighborhood $K$ of $\pi(\zeta)$. For all
large $m$, $\zeta_m\in\pi^{-1}(K)$, which is compact. If a subsequence
$\zeta_{m_j}$ converges to $\eta$, continuity of $\Psi$ gives
$\Psi(\eta) = \lim_{j\to\infty}\Psi(\zeta_{m_j}) = \Psi(\zeta).$
Injectivity implies $\eta=\zeta$. Thus every convergent subsequence
has limit $\zeta$, and compactness gives $\zeta_m\to\zeta$. Hence
$\Psi^{-1}$ is continuous.
\end{proof}

\begin{corollary}
\label{cor:positive-degeneration-cauchy-subsequence}
Let $\{[f_k]\}\subset\mathcal S_D$ satisfy
$h_1(f_k)\to+\infty$, and suppose that
$h_j(f_k)/h_1(f_k)$ converges to a positive number for every
$2\leq j\leq D-1$. Then $\{[f_k]\}$ has a
$\rho_D$-Cauchy subsequence.
\end{corollary}

\begin{proof}
The projective critical heights converge in the positive locus.
Properness of $\pi$ gives a subsequence, still denoted by
$\{[f_k]\}$, such that
$\hat\pi([f_k])\to\zeta$ for some
$\zeta\in\mathbb P\mathcal{ST}_D^*$.

Choose a normalized point $p_\zeta$ above $\zeta$ and set
$g_k=h_1(f_k)\cdot p_\zeta$. The sequence $\{g_k\}$ is
$\rho_D$-Cauchy by Lemma \ref{lem:radial-synchronization}, while
Proposition \ref{prop:asymptotic-model-collapse} gives
$\rho_D(f_k,g_k)\to0$. Hence the selected subsequence of
$\{f_k\}$ is $\rho_D$-Cauchy.
\end{proof}

\subsection{Proof of Theorem \ref{thm:completion}}
\label{subsec:proof-completion-theorem}

\begin{proof}[Proof of Theorem \ref{thm:completion}]
Proposition \ref{prop:quadratic-complete} proves completeness when
$D=2$.

Suppose that $D\geq3$. For any
$\zeta\in\mathbb P\mathcal{ST}_D^*$, Proposition
\ref{prop:boundary-class} produces a point
$\Psi(\zeta)\in\mathcal C_D^{\infty,+}$ in the metric boundary.
Hence $(\mathcal S_D,\rho_D)$ is incomplete.

Finally, Proposition \ref{prop:boundary-homeomorphism} gives the
natural homeomorphism
$
\mathcal C_D^{\infty,+}
\cong
\mathbb P\mathcal{ST}_D^*.
$
\end{proof}

\section{Comparison with the tree compactification}
\label{sec:tree-comparison}

In this section, we compare the metric geometry of the shift locus
with the DeMarco--McMullen polynomial-tree compactification. We first
prove Corollary \ref{cor:boundary-tree-comparison} in Section \ref{sec:pf-cor1.3}. We then
characterize the degenerations that escape to infinite
$\rho_D$-distance in Section \ref{sec:bdd-ns-degen} and prove Theorem \ref{thm:tree-infinite-distance} in Section \ref{sec:pf-thm1.4}.

\subsection{The positive-height boundary and polynomial trees}
\label{sec:pf-cor1.3}

Assume that $D\geq3$. Let $\Psi$ be the homeomorphism from
$\mathbb P\mathcal{ST}_D^*$ onto $\mathcal C_D^{\infty,+}$ given by
Proposition \ref{prop:boundary-homeomorphism}, and let
$q\colon\mathbb P\mathcal{ST}_D^*\to
\mathbb P\mathcal{ST}_D$ be the natural projection to the
projectivized polynomial-tree space. By
\cite[Theorem 1.7]{DeMarco11}, the map $q$ is continuous and
simplicial.

Define
$\Theta:=q\circ\Psi^{-1}\colon
\mathcal C_D^{\infty,+}\to\mathbb P\mathcal{ST}_D$.

\begin{proof}[Proof of Corollary \ref{cor:boundary-tree-comparison}]
The map $\Theta$ is continuous.
Let $x\in\mathcal C_D^{\infty,+}$ be represented by a
$\rho_D$-Cauchy sequence $\{[f_k]\}$. By Lemma
\ref{lem:positive-cauchy-model-limit}, there is a unique
$\zeta\in\mathbb P\mathcal{ST}_D^*$ such that
$\hat\pi([f_k])\to\zeta$. Since $x=\Psi(\zeta)$ and the
projectivized polynomial tree of $f_k$ is
$q(\hat\pi([f_k]))$, continuity of $q$ gives
$q(\hat\pi([f_k]))\to\Theta(x)$.

Finally, if $\tau\in\mathbb P\mathcal{ST}_D$, then
$\Theta^{-1}(\tau)=\Psi(q^{-1}(\tau))$. Since $\Psi$ is a
homeomorphism, $\Theta^{-1}(\tau)$ is naturally identified with
$q^{-1}(\tau)$.
\end{proof}

\subsection{Bounded distance for non-singular degenerations}
\label{sec:bdd-ns-degen}

We now prove the finite-distance direction of Theorem
\ref{thm:tree-infinite-distance}. For $s>0$, set
$\Delta_s:=(s,\ldots,s) \in \mathcal{H}_D$ and
$\mathcal R_s:=\mathcal G^{-1}(\Delta_s)$. We call
$\mathcal R_s$ the \emph{root fiber} at height $s$. By
\cite[Lemma 3.3]{DeMarco11}, $\mathcal R_s$ is connected.

We first observe that root fibers collapse at large height.

\begin{lemma}
\label{lem:root-fiber-collapse}
As $s\to+\infty$,
$\operatorname{diam}_{\rho_D}(\mathcal R_s)\to0$.
\end{lemma}

\begin{proof}
Suppose otherwise. Then there are $s_k\to+\infty$,
$f_k,g_k\in\mathcal R_{s_k}$, and $\varepsilon>0$ such that
$\rho_D(f_k,g_k)\geq\varepsilon$. Since each root fiber is connected,
all points of $\mathcal R_{s_k}$ have the same image under $\hat\pi$,
namely the root model in $\mathbb P\mathcal{ST}_D^*$. Proposition
\ref{prop:asymptotic-model-collapse} therefore gives
$\rho_D(f_k,g_k)\to0$, a contradiction.
\end{proof}

Now we construct a controlled pure-height path toward the root fiber.
Rather than invoking a stronger version of the pushing-up deformation,
we lift the desired affine height segment chamber by chamber, using the
openness of the marked critical heights map together with the closed
stretching sections constructed above.

\begin{lemma}
\label{lem:generic-root-access}
Let $[f]\in\mathcal S_D^g$ and set $s:=h_1(f)$. Write
$r_j=h_j(f)/s$ for $2\leq j\leq D-1$. There is an admissible
pure-height path $\alpha$ from $[f]$ to a point of $\mathcal R_s$
whose projective critical heights satisfy
$r_j(t)=(1-t)r_j+t, 2\leq j\leq D-1, 0\leq t\leq1.$
More precisely, $\alpha$ is a finite concatenation of pure-height
pieces with generic interiors.
\end{lemma}

\begin{proof}
Let
$x(t)=\bigl(1,(1-t)r_2+t,\ldots,(1-t)r_{D-1}+t\bigr).$
We first note that this segment meets the non-generic locus at only
finitely many parameters.  Indeed, for $i\neq j$ and $n\in\mathbb Z$,
the equation $x_i(t)=D^n x_j(t)$ has at most one solution unless it
holds identically, which is excluded by the genericity of the initial
height vector.  Only finitely many powers $D^n$ occur along the
compact segment.  Thus there is a finite subdivision
$0=t_0<t_1<\cdots<t_N=1$
such that $x((t_{\nu-1},t_\nu))$ lies in one generic chamber for every
$\nu$.

We construct the lift inductively.  On the first interval, independent
stretching in the generic component containing $[f]$ gives the unique
pure-height lift of $x(t)$.  The fixed-base closed stretching
construction extends this lift continuously to $t_1$; denote the
endpoint by $p_1$.

Suppose a lift has been constructed through $t_\nu$, with endpoint
$p_\nu$ and $h_1(p_\nu)=s$.  We claim that there is a generic component
$\mathcal U_{\nu+1}$ whose closure contains $p_\nu$ and whose
projective critical height image contains
$x((t_\nu,t_{\nu+1}))$.  Work on a critically marked lift of
$p_\nu$.  By \cite[Section 2.3]{DeMarco11}, the marked critical heights
map is open on the marked shift locus. Choose a nested neighborhood basis $\{V_m\}$ at the marked lift of
$p_\nu$.  Since the image of each $V_m$ is an
open neighborhood of the marked height vector $s\,x(t_\nu)$, and
$x(u)\to x(t_\nu)$ as $u\downarrow t_\nu$, we may choose
$u_m\in(t_\nu,t_{\nu+1})$ with $u_m\downarrow t_\nu$ and a
marked polynomial $p_{\nu,m}\in V_m$ whose critical height vector is
exactly $s\,x(u_m)$.  Then $p_{\nu,m}\to p_\nu$.  By local finiteness
of $\mathbb P\mathcal{ST}_D^*$, only finitely many global generic
components meet a sufficiently small neighborhood of $p_\nu$, as in
the proof of Lemma \ref{lem:local-sequential-collapse}.  After passing
to a subsequence, these points therefore lie in one fixed global
generic component $\mathcal U_{\nu+1}$.  Thus
$p_\nu\in\overline{\mathcal U_{\nu+1}}$.  Moreover, by the description
of the open stretching simplices in \cite[Section 7.1]{DeMarco11},
the projective critical height image of
$\mathcal U_{\nu+1}$ is the entire connected generic chamber
containing $x(u_m)$, hence it contains
$x((t_\nu,t_{\nu+1}))$.

Apply Lemma \ref{lem:finite-length-access} directly to $p_\nu$ and
$\mathcal U_{\nu+1}$, choosing the reference twist torus at largest
critical height $s$. The resulting closed-stretching section passes
through $p_\nu$. Restricting it to the segment
$x([t_\nu,t_{\nu+1}])$ produces a continuous lift with generic
interior in $\mathcal U_{\nu+1}$ and pure height on that interior.
This continues the path through $t_{\nu+1}$.

Iterating over the finite subdivision gives a continuous path
$\alpha$ from $[f]$ whose normalized critical height vector is exactly
$x(t)$ and which is a finite concatenation of pure-height pieces with
generic interiors.  At $t=1$ all critical heights equal $s$, so
$\alpha(1)\in\mathcal R_s$.  Hence $\alpha$ is admissible and has all
the asserted properties.
\end{proof}

\begin{lemma}
\label{lem:bounded-root-access}
Let $D\geq3$. For every $c>0$, there is $C_c>0$ with the following
property. If $[f]\in\mathcal S_D$ satisfies $s:=h_1(f)\geq1$ and
$
\frac{h_{D-1}(f)}{h_{D-2}(f)}\geq c,
$
then there is $p^\circ\in\mathcal R_s$ such that
$
\rho_D([f],p^\circ)\leq C_c.
$
\end{lemma}

\begin{proof}
Since the ratio in the hypothesis is at most $1$, assume
$0<c\leq1$.

Suppose first that $[f]$ is generic, and set
$r_j:=h_j(f)/s$. Let $\alpha$ be the path of Lemma
\ref{lem:generic-root-access}. Along this path,
$r_j(t)=(1-t)r_j+t.$
Since
$r_{D-1}/r_{D-2}\geq c$, we also have
$\frac{r_{D-1}(t)}{r_{D-2}(t)} \geq c$
for every $t$. Moreover, as observed in Lemma
\ref{lem:nonsingular-diagonal-path}, after division by
$r_{D-1}(t)$ every coordinate of the corresponding rose length
function is monotone. Since $\alpha$ is pure height,
its $\rho_D$-length is exactly the entropy length of the associated
normalized base rose path. Lemma
\ref{lem:nonsingular-diagonal-path} therefore gives
$
L_D(\alpha)\leq C_c.
$

Now let $[f]$ be non-generic. Choose a generic component whose closure
contains $[f]$. By the closed stretching construction in Lemma
\ref{lem:finite-length-access}, one may enter this component while
keeping $h_1=s$ and changing the projective critical heights by an
arbitrarily small amount. Since $\ell_B([f])$ has strictly positive coordinates, smoothness of
$Q$ and of the entropy metric near $\hat\ell_B([f])$ allows the access
curve to be chosen with length less than $1$. Choose its generic endpoint $[f']$ so close to $[f]$ that
$\frac{h_{D-1}(f')}{h_{D-2}(f')} \geq \frac c2.$
Applying the generic case to $[f']$ gives
$p^\circ\in\mathcal R_s$ with
$
\rho_D([f'],p^\circ)\leq C_{c/2}.
$
Thus
$\rho_D([f],p^\circ)\leq1+C_{c/2}$, which proves the lemma.
\end{proof}

\begin{proposition}
\label{prop:nonsingular-bounded-distance}
Let $D\geq3$, and let $\{[f_k]\}\subset\mathcal S_D$ satisfy
$h_1(f_k)\to+\infty$ and
$h_{D-1}(f_k)\asymp h_{D-2}(f_k)$. Then, for every fixed
$[f_0]\in\mathcal S_D$, there is $C>0$ such that
$\rho_D([f_0],[f_k])\leq C$ for all $k$.
\end{proposition}

\begin{proof}
Set $s_k:=h_1(f_k)$. After discarding finitely many terms, there is
$c>0$ such that
$\frac{h_{D-1}(f_k)}{h_{D-2}(f_k)}\geq c.$
By Lemma \ref{lem:bounded-root-access}, choose
$p_k^\circ\in\mathcal R_{s_k}$ such that
$\rho_D([f_k],p_k^\circ)\leq C_c$.

Fix any $p^\circ\in\mathcal R_1$. Lemma
\ref{lem:radial-synchronization}, with a compact positive set containing
$r(p^\circ)$, shows directly that $\{s_k\cdot p^\circ\}$ is a
$\rho_D$-Cauchy sequence and hence is bounded. Since both
$p_k^\circ$ and $s_k\cdot p^\circ$ lie in the root fiber
$\mathcal R_{s_k}$, Lemma \ref{lem:root-fiber-collapse} gives
$\rho_D(p_k^\circ,s_k\cdot p^\circ) \leq \operatorname{diam}_{\rho_D}(\mathcal R_{s_k}) \to 0.$
Therefore
$
\rho_D([f_0],[f_k])
\leq
\rho_D([f_0],s_k\cdot p^\circ)
+
\operatorname{diam}_{\rho_D}(\mathcal R_{s_k})
+
C_c,
$
and the right-hand side is uniformly bounded.
\end{proof}

\subsection{Proof of Theorem \ref{thm:tree-infinite-distance}}
\label{sec:pf-thm1.4}

\begin{proof}[Proof of Theorem \ref{thm:tree-infinite-distance}]
Suppose first that $D=2$. Since $h_1(f_k)\to+\infty$, by Corollary
\ref{cor:singular-degenerations}, we have 
$
\rho_2([f_0],[f_k])\to+\infty.
$

Now let $D\geq3$. If
$
\frac{h_{D-1}(f_k)}{h_{D-2}(f_k)}\to0,
$
Corollary \ref{cor:singular-degenerations} gives
$\rho_D([f_0],[f_k])\to+\infty$.

Conversely, suppose that
$h_{D-1}(f_k)/h_{D-2}(f_k)$ does not converge to zero. Passing to a
subsequence, there is $c>0$ such that
$\frac{h_{D-1}(f_k)}{h_{D-2}(f_k)}\geq c.$
Proposition \ref{prop:nonsingular-bounded-distance} gives a uniform
bound for $\rho_D([f_0],[f_k])$ along this subsequence. Hence the
original sequence of distances cannot tend to $+\infty$. This proves
the converse.
\end{proof}

\bibliographystyle{siam}
\bibliography{SL_rev_arXiv}
\end{document}